\documentclass{birkjour}

\usepackage{xcolor}
\definecolor{webgreen}{rgb}{0,.5,0}
\definecolor{webbrown}{rgb}{.6,0,0}
\definecolor{RoyalBlue}{cmyk}{1, 0.50, 0, 0}
\usepackage[colorlinks=true, breaklinks=true, urlcolor=webbrown, linkcolor=RoyalBlue, citecolor=webgreen,backref=page]{hyperref}
\usepackage{amsmath,amssymb,amsthm,upgreek}
\usepackage{tikz, epsfig, graphicx, subfigure}
\usepackage{verbatim}

\newcommand{\R}{{\mathbb R}}
\newcommand{\N}{{\mathbb N}}
\newcommand{\C}{{\mathbb C}}

\newcommand{\Z}{{\mathbb Z}}

\newcommand{\ic}{{\mathrm i}}
\newcommand{\dd}{{\mathrm d}}

\newcommand{\im}{\mathrm{Im}}
\newcommand{\re}{\mathrm{Re}}

\newcommand{\RS}{\boldsymbol{\mathfrak{R}}}
\newcommand{\bd}{\boldsymbol{\Delta}}
\newcommand{\ualpha}{\boldsymbol\upalpha}
\newcommand{\ubeta}{\boldsymbol\upbeta}

\newcommand{\z}	{{\boldsymbol z}}
\newcommand{\tr} {{\boldsymbol t}}
\newcommand{\s}	{{\boldsymbol s}}

\newcommand{\qandq}{\quad \text{and} \quad}
\newcommand{\qasq}{\quad \text{as} \quad }

\newcommand{\rhy}   {\textnormal{RHP}-${\boldsymbol Y}$}
\newcommand{\rhx}   {\textnormal{RHP}-${\boldsymbol X}$}
\newcommand{\rhn}   {\textnormal{RHP}-${\boldsymbol N}$}
\newcommand{\rhz}   {\textnormal{RHP}-${\boldsymbol Z}$}
\newcommand{\rhpi}   {\textnormal{RHP}-${\boldsymbol P_{a_i}}$}
\newcommand{\rhp}   {\textnormal{RHP}-${\boldsymbol P}_{a_1}$}
\newcommand{\rhpo}   {\textnormal{RHP}-${\boldsymbol P}_0$}
\newcommand{\rhpsiA}   {\textnormal{RHP}-${\boldsymbol \Psi_\alpha}$}
\newcommand{\rhpsiS}   {\textnormal{RHP}-${\boldsymbol \Psi_{s_1, s_2}}$}

 \newtheorem{thm}{Theorem}[section]
 
 \newtheorem{lem}[thm]{Lemma}
 \newtheorem{prop}[thm]{Proposition}
 \theoremstyle{definition}
 \newtheorem*{defn}{Definition}
 \theoremstyle{remark}

 \numberwithin{equation}{section}

\begin{document}

%
%
%
%
%
%
%
%
%

\title[Jacobi-type Polynomials on a Cross]
 {Asymptotics of Polynomials Orthogonal on a Cross with a Jacobi-type Weight}

\author{Ahmad Barhoumi}

\address{Department of Mathematical Sciences\\
 Indiana University-Purdue University Indianapolis\\
 402~North Blackford Street, Indianapolis, IN 46202}

\email{\href{mailto:abarhoum@iupui.edu}{abarhoum@iupui.edu}}

\author{Maxim L. Yattselev}
\address{Department of Mathematical Sciences\\
 Indiana University-Purdue University Indianapolis\\
 402~North Blackford Street, Indianapolis, IN 46202}

\email{\href{mailto:maxyatts@iupui.edu}{maxyatts@iupui.edu}}

\thanks{The research was supported in part by a grant from the Simons Foundation, CGM-354538.}



\date{\today}

\begin{abstract}
We investigate asymptotic behavior of polynomials \( Q_n(z) \) satisfying non-Hermitian orthogonality relations
\[
\int_\Delta s^kQ_n(s)\rho(s)\dd s =0, \quad k\in\{0,\ldots,n-1\},
\]
where \( \Delta := [-a,a]\cup [-\ic b,\ic b] \), \( a,b>0 \), and \( \rho(s) \) is a Jacobi-type weight. 
\end{abstract}

\maketitle

\section{Introduction}

Let \( a,b> 0 \) be fixed. Set
\begin{equation}
\label{Delta}
\Delta := [-a,a]\cup [-\ic b,\ic b] \qandq \Delta^\circ := \Delta\setminus\{0,a_1,a_2,a_3,a_4\},
\end{equation}
where we put \( a_1 = -a_3 = a \) and  \( a_2=-a_4 = \ic b \). Denote by \( \Delta_i \) (resp. \( \Delta_i^\circ \)), \( i\in\{1,2,3,4\} \), the closed (resp. open) segment joining the origin and \( a_i \), which we orient towards the origin. In this work we are interested in strong asymptotics of polynomials \( Q_n(z) \), \( \deg(Q_n)\leq n \), satisfying orthogonality relations
\begin{equation}
\label{orthogonality}
\int_\Delta s^kQ_n(s)\rho(s)\dd s =0, \quad k\in\{0,\ldots,n-1\},
\end{equation}
where \( \Delta \) inherits its orientation from the segments \( \Delta_i \) and \( \rho(s) \) is a certain weight function on \( \Delta \). Orthogonality relations \eqref{orthogonality} are non-Hermitian as \( s^k \) is not conjugated. Hence, there are no a priori reasons to assume that \( \deg(Q_n) = n \). In what follows, we shall understand that \( Q_n(z) \) stands for the monic polynomial of minimal degree satisfying \eqref{orthogonality}. The weight functions we are interested in are holomorphic perturbations of the power functions. More precisely, we define the following nested sequence of classes of weights. 

\begin{defn}
Let \( \ell \) be a positive integer or infinity. We shall say that a function \( \rho(s) \) on \( \Delta \) belongs to the class \( \mathcal W_\ell \) if
\begin{itemize}
\item[(i)]  \( \rho_i(s):=\rho_{|\Delta_i^\circ}(s) \) factors as a product \(  \rho_i(s) = \rho_i^*(s)(s-a_i)^{\alpha_i} \), where the function \( \rho_i^*(z) \) is non-vanishing and holomorphic in some neighborhood of \( \Delta_i \), \( \alpha_i>-1 \), and \( (z-a_i)^{\alpha_i} \) is a branch holomorphic across \( \Delta\setminus\{a_i\} \), \( i\in\{1,2,3,4\} \);
\item[(ii)] the ratio \( (\rho_1\rho_3)(z)/(\rho_2\rho_4)(z) \) is constant in some neighborhood of the origin (notice that each \( \rho_i(s) \) extends holomorphically to a neighborhood of the origin by (i));
\item[(iii)] it holds that \(\rho_1(0) + \rho_2(0) + \rho_3(0) + \rho_4(0) = 0\);
\item[(iv)] the quantities \( \rho_i^{(l)}(0)/\rho_i(0) \), \( 0\leq l<\ell \), do not depend on \( i\in\{1,2,3,4\} \).
\end{itemize}
\end{defn}

Observe that conditions (ii) and (iii) say that one of the functions \( \rho_i(z) \) is fully determined by the other three. In particular, it must hold that
\[
\rho_4(z) = -(\rho_1+\rho_2+\rho_3)(0)(\rho_2/\rho_1\rho_3)(0)(\rho_1\rho_3/\rho_2)(z).
\]
Notice also that \( \mathcal W_{\ell_1}\subset \mathcal W_{\ell_2} \) whenever \( \ell_2<\ell_1 \) and that \( \rho(s)\in\mathcal W_\infty \) if and only if there exists a function \( F(z) \), holomorphic in some neighborhood of \( \Delta\setminus\{a_1,a_2,a_3,a_4\} \), such that  \( \rho_i(s)=c_iF_{|\Delta_i^\circ}(s) \) for some constants \( c_i \) that add up to zero.

Holomorphy of the weights \( \rho_i(z) \) allows one to deform \( \Delta \) in \eqref{orthogonality} to any cross-like contour consisting of four arcs connecting the points \( a_i \) to the origin (some central point if the weight add up to zero in a neighborhood of the origin). Hence, the following question arises: \emph{which contour do we expect to attract the zeros of the polynomials \( Q_n(z) \) as \( n\to\infty\)?} This fundamental question in the theory of non-Hermitian orthogonal polynomials was answered by Herbert Stahl in \cite{St85,St85b,St86}. It turns out that the attracting contour is essentially characterized by having the smallest logarithmic capacity among all continua containing \( \{ a_1,a_2,a_3,a_4 \} \). It is also known from the works \cite{Grot30,Lav30} that this contour must consist of the orthogonal critical trajectories of the quadratic differential
\begin{equation}
\label{differential}
\mathcal D(z)\dd z^2 = \frac{(z-b_1)(z-b_2)\dd z^2}{(z^2-a^2)(z^2+b^2)}
\end{equation}
for some uniquely determined constants \( b_1,b_2 \). It can be readily verified that \( \Delta \) is the desired contour and \( b_1=b_2=0 \). In fact, the work of Stahl not only supplies us with the attracting contour, but also tell us that \( \frac1{\pi\ic}\sqrt{\mathcal D(s)}_+\dd s \) is the limiting distribution of zeros of \( Q_n(z) \), where the subscript \( + \) stands for the trace on the positive side of \( \Delta \) (according to the chosen orientation). 

Strong asymptotics of the polynomials \( Q_n(z) \) was considered as part of a study in \cite{Y15} under much more restrictive assumption \( \rho(s)=h(s)/w_+(s) \), where \( h(z) \) is a holomorphic and non-vanishing function is some neighborhood of \( \Delta \) and \( w(z) \) is defined in \eqref{w} further below. It is also worth pointing out that if the points \( \{ a_1,a_2,a_3,a_4 \} \) do not form a  cross with two symmetries, then the points \( b_1,b_2 \) in \eqref{differential} are distinct and the corresponding minimal capacity contour consists of five arcs: one joining \( b_1 \) and \( b_2 \), two connecting \( b_1 \) to two points in \( \{ a_1,a_2,a_3,a_4 \} \), and two connecting \( b_2 \) to the other two points in \( \{ a_1,a_2,a_3,a_4 \} \). Non-Hermitian orthogonal polynomials on such a contour for a class of weights defined similarly to \( \mathcal W_1 \) are a particular example of polynomials studied in \cite{ApY15}.

This study is not only motivated by the authors' intrinsic interest in the behavior of non-Hermitian orthogonal polynomials, but also by the research program of determining the effective constants in the functional analog of Thue-Siegel-Roth theorem, see \cite{AptY16}  for a more detailed explanation. Briefly, let \( f(z)=\sum_{i=0}^\infty f_iz^{-i} \) be a convergent power series. The \( n \)-th diagonal Pad\'e approximant for \( f(z) \) is a rational function \( [n/n]_f(z) = p_n(z)/q_n(z) \) such that \(\deg(p_n),\deg(q_n)\leq n \) and
\[
(q_nf-p_n)(z) = \mathcal O\left(z^{-n-1}\right) \qasq z\to\infty.
\]
It is known that \( [n/n]_f(z) \) always exists and is unique even though there may be many pairs \( (p_n(z),q_n(z)) \) satisfying the above relation. It is customary then to write \( [n/n]_f(z) \) in the reduced form, that is, to take \( q_n(z) \) to be the solution of the smallest degree, which is also normalized to be monic.  If \( \deg(q_n)=n \) and \( (q_nf-p_n)(z) \sim z^{-n-m_n-1} \) at infinity, then \( (f-[n/n]_f)(z) \) vanishes there with order \( 2n+m_n+1 \). Kolchin \cite{Kolch59} conjectured that for a class of functions \( f(z) \), including algebraic ones, it holds that \( m_n\leq \epsilon n + C_{\epsilon,f} \) for any \( \epsilon>0 \) and some constant \(  C_{\epsilon,f}>0 \) that depends on \( \epsilon \) and \( f(z) \). This conjecture was proven in \cite{Uch61,ChudChud84} with various degrees of specificity about the constants. Around  1984, it was additionally conjectured by A.A.~Gonchar that for an algebraic function \( f(z) \) it must, in fact, hold that \( m_n \leq C_f \). Given an algebraic function \( f(z) \), determining the constant \( C_f \) is the overarching goal which the authors have their eyes on. More precisely, given a germ of an algebraic function \( f(z) \) at infinity, and assuming that \( f(z) \) has no polar singularities and the branching singularities have integrable order, this germ can be written as a Cauchy integral of its jump across Stahl's (minimal capacity) contour. The polynomials \( q_n(z) \) are then orthogonal with respect to the jump of the germ on this contour. Moreover, it also can be shown that \( q_{n+j}(z)=q_n(z) \) for \( j\leq m_n \). Thus, to determine the size of \( m_n \) one can study the strong asymptotic behavior of \( q_n(z) \). In \cite{ApY15}, a generic situation was considered when Stahl's contour does not have more than three Jordan arcs meeting at any point. In this case \( m_n\leq g \), where \( g \) is the genus of the Riemann surface corresponding to Stahl's contour. In the present work we consider the first non-generic model case when there is a point common to four Jordan arcs.

\section{Statement of Results}

The functions describing the asymptotics of the polynomials \( Q_n(z) \) are constructed in three steps, carried out in Sections~\ref{ssec:map}-\ref{ssec:theta}, and naturally defined on a Riemann surface corresponding to \( \Delta \) that is introduced in Section~\ref{ssec:surface}. The main results of this work are stated in Sections~\ref{ssec:asymp} and~\ref{ssec:pade}. A more detailed description of the material in Sections~\ref{ssec:surface}--\ref{ssec:theta} can be found in~\cite{Zver71}.

\subsection{Riemann Surface}
\label{ssec:surface}

Let \( \Delta=\cup_{i=1}^4\Delta_i \) be given by \eqref{Delta}. Set
\begin{equation}
\label{w}
w(z) := \sqrt{(z^2-a^2)(z^2+b^2)}, \quad z\in\C\setminus\Delta,
\end{equation}
to be the branch normalized so that \( w(z) = z^2 + \mathcal O(z) \) as \( z\to\infty \).  Denote by $\RS$ the Riemann surface of $w(z)$ realized as a two-sheeted ramified cover of $\overline\C$ constructed in the following manner. Two copies of $\overline\C$ are cut along each arc $\Delta_i$. These copies are glued together along the cuts in such a manner that the right (resp. left) side of the arc $\Delta_i$ belonging to the first copy, say $\RS^{(0)}$, is joined with the left (resp. right) side of the same arc $\Delta_i$ only belonging to the second copy, $\RS^{(1)}$. 
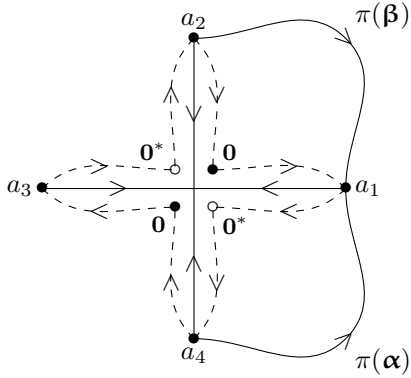
\begin{figure}[!ht]

\begin{tikzpicture}
\draw (2, 0) -- (-2,0);
\draw (0, 2) -- (0, -2);
\node[rotate = 90] at (0, 1){$<$};
\node[rotate = 90] at (0, -1){$>$};
\node at (1, 0){$<$};
\node at (-1, 0){$>$};
\node at (2, 0){\textbullet}; \node[right] at (2, 0){$a_1$};
\node at (-2, 0){\textbullet}; \node[left] at (-2, 0){$a_3$};
\node at (0, 2){\textbullet}; \node[above] at (0, 2){$a_2$};
\node at (0, -2){\textbullet}; \node[below] at (0, -2){$a_4$};
%
\draw (2, 0) to[out= 90, in = -45](2, 2) to[out = 135, in = 0] (0, 2); \node[rotate = 135] at (2, 2){$<$}; \node[above right] at (2, 2){$\pi(\ubeta)$};

\draw (2, 0) to[out= -90, in = 45](2, -2) to[out = -135, in = 0](0, -2); \node[rotate = 45] at (2, -2){$>$}; \node[below right] at (2, -2){$\pi(\ualpha)$};
%
\node at (0.25, 0.25){\textbullet}; \node[above right] at (0.25, 0.25){\small $\boldsymbol 0$};
\node at (-0.25, -0.25){\textbullet}; \node[below left] at (-0.25, -0.25){\small $\boldsymbol 0$}; 
\node at (0.25, -0.25){$\circ$}; \node[below right] at (0.25, -0.25){\small $\boldsymbol 0^*$};
\node at (-0.25, 0.25){$\circ$}; \node[above left] at (-0.25, 0.25){\small $\boldsymbol 0^*$};
%
\draw[style = dashed] (2, 0) to[out = 135, in = 0] (0.25, 0.25); \node at (1.25, 0.3){$>$};
\draw[style = dashed] (2, 0) to[out = -135, in = 0] (0.25, -0.25); \node at (1.25, -0.3){$<$};
\draw[style = dashed] (0, 2) to[out = -45, in = 90] (0.25, 0.25); \node[rotate = 90] at (0.3, 1.25){$<$};
\draw[style = dashed] (0, 2) to[out = -135, in = 90] (-0.25, 0.25); \node[rotate = 90] at (-0.3, 1.25){$>$};
\draw[style = dashed] (-2, 0) to[out = 45, in = 180] (-0.25, 0.25); \node at (-1.25, 0.3){$>$};
\draw[style = dashed] (-2, 0) to[out = -45, in = 180] (-0.25, -0.25); \node at (-1.25, -0.3){$<$};
\draw[style = dashed] (0, -2) to[out = 45, in = -90] (0.25, -0.25); \node[rotate = 90] at (0.3, -1.25){$<$};
\draw[style = dashed] (0, -2) to[out = 135, in = -90] (-0.25, -0.25); \node[rotate = 90] at (-0.3, -1.25){$>$};
\end{tikzpicture}
\caption{\small The arcs \( \Delta_i \) together with their orientation (solid lines), a schematic representation of the arcs \( \bd_i = \pi^{-1}(\Delta_i) \) (dashed lines) as viewed from \( \RS^{(0)} \), and the chosen homology basis \( \{\ualpha,\ubeta\} \) projected down from \( \RS^{(0)} \).}
\label{f:cross}
\end{figure}
We denote by $\pi$ the canonical projection $\pi:\RS\to\overline\C$ and define \( \bd:=\pi^{-1}(\Delta) \), \( \bd_i:=\pi^{-1}(\Delta_i) \), \( i\in\{1,2,3,4\} \). Then \( \bd \) is a curve on \( \RS \) that intersects itself exactly twice (once at each point on top of the origin), see Figures~\ref{f:cross} and~\ref{f:basis}. We orient \( \bd \) so that \( \RS^{(0)} \) remains on the left when $\bd$ is traversed in the positive direction. We shall denote by \( z^{(k)} \), \( k\in\{0,1\} \), the point on \( \RS^{(k)} \) with the canonical projection \( z \)  and designate the symbol $\cdot^*$ to stand for the conformal involution that sends $z^{(k)}$ into $z^{(1-k)}$, $k\in\{0,1\}$. We use bold lower case letters such as $\z,\tr,\s$ to indicate points on $\RS$ with the canonical projections $z,t,s$. Since \( \RS \) has genus \( 1 \), any homology basis on \( \RS \) consists of only two cycles. In what follows, we choose cycle \( \ualpha \) (resp. \(\ubeta \)) to be involution-symmetric and such that \( \pi(\ualpha) \) (resp. \(\pi(\ubeta) \)) is a rectifiable Jordan arc joining \( a_1 \) and \(a_2\) (resp. \( a_4 \) and \(a_1 \)), that belongs to fourth (resp. first) quadrant and does not intersect \( \Delta^\circ \cup \{0\} \), see Figure~\ref{f:cross}. We orient these cycles towards \( \boldsymbol a_1 \) on \( \RS^{(0)} \) and therefore away from \( \boldsymbol a_1 \) on \( \RS^{(1)} \), see Figure~\ref{f:basis}.

\begin{figure}[!ht]

\begin{tikzpicture}[scale = 1.2]
\fill[black!30] (0, 0) -- (-2, 1) -- (0, 2) -- (4, 2) -- (4, 0) -- (2, -1) -- cycle;
\fill[black!30] (-2, 1) -- (-4, 0) -- (-4, 2) -- cycle;
\fill[black!30] (2, -1) -- (0, -2) -- (4, -2) -- cycle;
%
\node at (4, 2){\textbullet};\node[above right] at (4, 2){$\boldsymbol a_1$};
\node at (4, -2){\textbullet};\node[below right] at (4, -2){$\boldsymbol a_1$};
\node at (-4, 2){\textbullet};\node[above left] at (-4, 2){$\boldsymbol a_1$};
\node at (-4, -2){\textbullet};\node[below left] at (-4, -2){$\boldsymbol a_1$};
\node at (-4, 0){\textbullet};
\node[left] at (-4, 0){$\boldsymbol a_2$};
\node at (4, 0){\textbullet};
\node[right] at (4, 0){$\boldsymbol a_2$};
\node at (0, 0){\textbullet }; \node [above] at (0, 0){$\boldsymbol a_3$};
\node at (0, 2){\textbullet}; \node [above] at (0, 2){$\boldsymbol a_4$};
\node at (0, -2){\textbullet}; \node [below] at (0, -2){$\boldsymbol a_4$};
\node at (-2, 1) {\textbullet}; \node [below] at (-2, 0.9){${\boldsymbol 0^*}$};
\node at (2, -1) {\textbullet}; \node [below] at (2, -1.1){${\boldsymbol 0}$};
\draw[line width = 0.4 mm] (-2, 1) -- (-4, 2); \node[rotate = 150] at (-3, 1.5){$<$};\node[above right] at (-3, 1.5){$\bd_1$};
\draw[line width = 0.4 mm] (4, -2) -- (2, -1); \node [rotate = 150] at (3, -1.5){$<$};\node [above right] at (3, -1.5){$\bd_1$};
\draw[line width = 0.4 mm] (-4, 0) -- (-2, 1); \node[rotate = 30] at (-3, 0.5){$<$}; \node[below right] at (-3, 0.5){$\bd_2$};
\draw[line width = 0.4 mm] (2, -1) -- (4, 0); \node [rotate = 30] at (3, -0.5){$<$};\node [above left] at (3, -0.5){$\bd_2$};
\draw[line width = 0.4 mm] (0,0) -- (-2, 1); \node[rotate = -30] at (-1, 0.5){$<$}; \node[above right] at (-1, 0.5){$\bd_3$};
\draw[line width = 0.4 mm] (2, -1) -- (0, 0); \node [rotate = -30] at (1, -0.5){$<$};\node [above right] at (1, -0.5){$\bd_3$};
\draw[line width = 0.4 mm] (-2, 1)--(0, 2); \node[rotate = 210] at (-1, 1.5){$<$}; \node[above left] at (-1, 1.5){$\bd_4$};
\draw[line width = 0.4 mm] (2, -1) -- (0, -2); \node [rotate = 210] at (1, -1.5){$<$};\node [above left] at (1, -1.5){$\bd_4$};
\draw (-4, -2) -- (4, -2); \node at (-2, -2){$<$};\node at (2, -2){$<$};
\node[below] at (-2, -2){$\ualpha$};\node[below] at (2, -2){$\ualpha$};
\draw (-4, 2) -- (4, 2); \node at (-2, 2){$<$};\node at (2, 2){$<$};
\node[above] at (-2, 2){$\ualpha$};\node[above] at (2, 2){$\ualpha$};
\draw (-4, -2) -- (-4, 2); \node[rotate = 90] at (-4, -1){$<$};\node[rotate = 90] at (-4, 1){$<$};
\node[left] at (-4, -1){$\ubeta$};\node[left] at (-4, 1){$\ubeta$};
\draw (4, -2) -- (4, 2); \node[rotate = 90] at (4, -1){$<$};\node[rotate = 90] at (4, 1){$<$};
\node[right] at (4, -1){$\ubeta$};\node[right] at (4, 1){$\ubeta$};
\node at (-2.5, -0.5){$\RS^{(0)}$};
\node at (2.5, 0.5){$\RS^{(1)}$};
\end{tikzpicture}
\caption{\small Schematic representation of the surface \( \RS \) (shaded region represents \( \RS^{(1)} \)), which topologically is a torus, the arcs \( \bd_1,\bd_2,\bd_3,\bd_4 \), and the homology basis \(\ualpha,\ubeta\).}
\label{f:basis}
\end{figure}
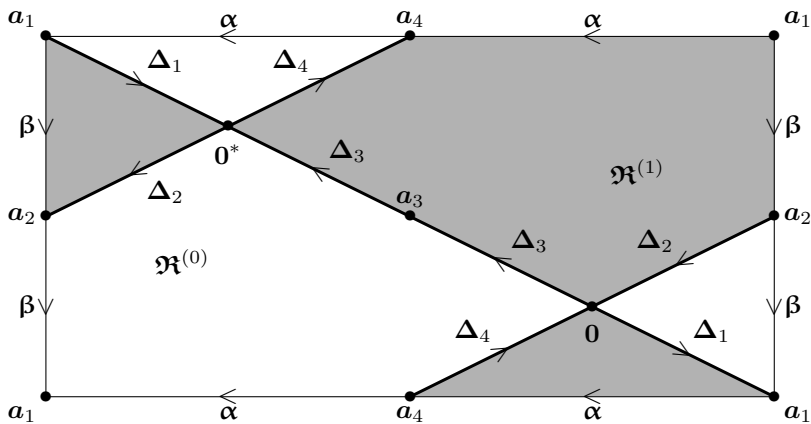

\subsection{Geometric Term}
\label{ssec:map}

The main goal of this subsection is to define the function \( \Phi(\z) \), see \eqref{Phi}, that will be responsible for the rate of growth of the polynomials \( Q_n(z) \) and is determined solely by the contour of orthogonality \( \Delta \).

With a slight abuse of notation, let us set
\[
w(\z) := (-1)^kw(z), \quad \z\in\RS^{(k)}\setminus\bd, \quad k\in\{0,1\},
\]
which we then extend by continuity to \( \bd \). Clearly, \( w(\z) \) is a meromorphic function on \( \RS \) with simple zeros at the ramification points of \( \RS \), double poles at \( \infty^{(0)} \) and \( \infty^{(1)} \), and otherwise non-vanishing and finite. Thus,
\begin{equation}
\label{Omega}
\Omega(\z) := \left(\oint_{\ualpha}\frac{\dd s}{w(\s)}\right)^{-1}\frac{\dd z}{w(\z)}
\end{equation}
is the holomorphic differential on \( \RS \) normalized to have unit period on \( \ualpha \). In this case it was shown by Riemann that the constant
\begin{equation}
\label{B}
\mathsf B := \oint_{\ubeta}\Omega
\end{equation}
has positive purely imaginary part. Further, since \( z/w(\z) \) has simple poles at the ramification point of \( \RS \), simple zeros at \( \infty^{(0)} \) and \( \infty^{(1)} \), and behaves like \( 1/z \) around \(\infty^{(0)} \), the differential
\[
G(\z) := \frac{z\dd z}{w(\z)}
\]
is meromorphic on \( \RS \) having two simple poles at \( \infty^{(1)} \) and \( \infty^{(0)} \) with respective residues \( 1 \) and \( -1 \). \( G(\z) \) is also distinguished by having a purely imaginary period on any cycle on \( \RS \). Indeed, it is enough to verify this claim on the cycles of any homology basis. To this end, define
\begin{equation}
\label{constants}
\omega:=-\frac1{2\pi\ic}\oint_{\ubeta}G \quad \mbox{and} \quad \tau:=\frac1{2\pi\ic}\oint_{\ualpha}G.
\end{equation}
By deforming $\ualpha$ (resp. $\ubeta$) into $-\bd_1-\bd_4$ (resp. $\bd_1+\bd_2$) and using the symmetry $G (\z^*) = - G(\z)$, one gets that 
\begin{equation}
\label{omtau12}
\omega = \tau = 
 \frac1{4\pi\ic}\oint_\Gamma\frac{z\dd z}{w(z)} = \dfrac{1}{2},
\end{equation}
where \( \Gamma \) is any positively oriented rectifiable Jordan curve encircling \( \Delta \), which does verify the claim about \( G(\z) \) having purely imaginary periods. Let
\begin{equation}
\label{Phi}
\Phi(\z):= \exp\left\{\int_{\boldsymbol a_3}^\z G\right\}, \quad \z\in\RS_{\ualpha,\ubeta}\setminus\big\{\infty^{(0)},\infty^{(1)}\big\},
\end{equation}
where \( \RS_{\ualpha,\ubeta}:=\RS\setminus\{\ualpha,\ubeta\} \) and the path of integration lies entirely in \( \RS_{\ualpha,\ubeta}\setminus\big\{\infty^{(0)},\infty^{(1)}\big\} \). The function $\Phi(\z)$ is holomorphic and non-vanishing on $\RS_{\ualpha,\ubeta}$ except for a simple pole at $\infty^{(0)}$ and a simple zero at $\infty^{(1)}$. Furthermore, it possesses continuous traces on both sides of each cycle of the canonical basis that satisfy\footnote{Here and in what follows we state jump relations understanding that they hold outside the points of self-intersection of the considered arcs.}
\begin{equation}
\label{Phi-jump}
\Phi_+(\s) = -\Phi_-(\s), \quad \s \in \ualpha \cup \ubeta,
\end{equation}
by \eqref{constants}--\eqref{omtau12}. It is not a difficult computation to check that \( \Phi(\z)\Phi(\z^*) \equiv 1 \) and
\begin{equation}
\label{Green}
\big|\Phi(\z)\big|=\exp\left\{(-1)^kg_\Delta(z;\infty)\right\}, \quad \z\in\RS^{(k)},
\end{equation}
\( k\in\{0,1\} \), where \( g_\Delta(z;\infty) \) is the Green function for \( \overline\C\setminus\Delta \) with pole at \( \infty \).\footnote{\( g_\Delta(z;\infty) \) is equal to zero on \( \Delta \), is positive and harmonic in \( \C\setminus\Delta \), and satisfies \( g(z;\infty)=\log|z|+\mathcal O(1) \) as \(z\to\infty\).} In fact, the above properties allow us to verify that
\begin{equation}
\label{PhiSquared}
\Phi^2\big(z^{(k)}\big) = \frac2{a^2+b^2}\left(z^2 + \frac{b^2-a^2}2 +(-1)^k w(z)\right),
\end{equation}
\( k\in\{0,1\} \). In particular, this implies that the logarithmic capacity of \( \Delta \) is equal to \( \sqrt{a^2+b^2}/2 \) since
\begin{equation}
\label{log-cap}
\Phi\big(z^{(0)}\big) = \frac{-2z}{\sqrt{a^2+b^2}} + \mathcal O(1) \quad \text{as} \quad z\to\infty
\end{equation}
(the sign in \eqref{log-cap} is determined by the fact that \( \Phi(\boldsymbol a_3)=1 \) and \( \Phi(\z) \) is non-vanishing on \( \pi^{-1}((-\infty,-a)) \)). Observe also that a calculus level computation tells us that
\begin{equation}
\label{K0-formula}
\Phi({\boldsymbol 0}) = \overline{\Phi(\boldsymbol 0^*)} = 
  \exp \left\{ \ic \arctan \left(\frac{a}{b}\right) \right \},
\end{equation}
where the point \( \boldsymbol 0 \) and \( \boldsymbol 0^* \) are defined as on Figure~\ref{f:cross}.

\subsection{Szeg\H{o} Function}
\label{ssec:szego}

It is  known since the work of Szeg\H{o} that the finer details of the asymptotics of \( Q_n(z) \) are captured by the so-called Szeg\H{o} function, which depends only on the weight of orthogonality. Below, we construct this function for  \( \rho(s)\in\mathcal W_1 \). 

Given \( i\in\{1,2,3,4\} \), fix \( \log\rho_i(s) \) to be a branch continuous on \( \Delta_i\setminus\{a_i\} \), selected so that
\begin{equation}
\label{nu}
\nu := \frac1{2\pi\ic}\sum_{i=1}^4(-1)^i\log\rho_i(0) \quad \text{satisfies} \quad \mathrm{Re}(\nu)\in\left(-\frac12,\frac12\right].
\end{equation}
Further, it can be readily verified that we can set
\begin{equation}
\label{logw}
\log w_+(s) = \log|w_+(s)| + (-1)^i\frac{\pi\ic}2, \quad s\in\Delta_i^\circ,
\end{equation}
where, as usual, \( w_+(s) \) is the trace of \eqref{w} on the positive side of \( \Delta_i^\circ\) according to the chosen orientation. We also let \( \log(\rho_i w_+)(s) \) to stand for \( \log\rho_i(s) + \log w_+(s) \) with the just selected branches. Put
\begin{equation}
\label{szego}
S_\rho(\z) := \exp\left\{-\dfrac{1}{4 \pi \ic}\oint_{\bd}\log(\rho w_+)(s)\Omega_{\z,\z^*}(\s)\right\},
\end{equation}
where \( \Omega_{\z,\z^*}(\s) \) is the meromorphic differential with two simple poles at $\z$ and $\z^*$ with respective residues \( 1 \) and $-1$ normalized to have zero period on $\ualpha$. When \( \z \) does not lie on top of the point at infinity, it can be readily verified that
\begin{equation}
\label{cauchy-kernel}
\Omega_{\z,\z^*}(\s) = \dfrac{w(\boldsymbol z)}{s - z} \dfrac{\dd s}{w(\s)} - \left(\oint_{\ualpha} \dfrac{w(\boldsymbol z)}{t - z}\dfrac{\dd t}{w(\tr)} \right)\Omega(\s),
\end{equation}
where \( \Omega(\s) \) is the holomorphic differential \eqref{Omega}.

\begin{prop}
\label{prop:szego}
Let \( \rho(s)\in \mathcal W_1 \) and \( S_\rho(\z) \) be given by \eqref{szego}. Define
\begin{equation}
\label{crho}
c_\rho := \frac1{2\pi\ic}\oint_{\bd}\log(\rho w_+)\Omega.
\end{equation}
Then \( S_\rho(\z) \) is a holomorphic and non-vanishing function in \( \RS\setminus\{\bd\cup\ualpha\} \) with continuous traces on \( (\bd\cup\ualpha)\setminus\{\boldsymbol a_1,\boldsymbol a_2,\boldsymbol a_3,\boldsymbol a_4,\boldsymbol 0,\boldsymbol 0^*\} \) that satisfy
\begin{equation}
\label{S-jump}
S_{\rho+}(\s) = S_{\rho-}(\s)\left\{
\begin{array}{rl}
\displaystyle \exp\big\{2\pi\ic c_\rho\big\}, &  \s\in\ualpha, \medskip \\
1/(\rho w_+)(s), & \s\in\bd.
\end{array}
\right.
\end{equation}
It also holds that \( S_\rho(\z)S_\rho(\z^*)\equiv1 \) and \footnote{In what follows we write \( |g_1(z)|\sim|g_2(z)| \) as \( z\to z_0 \) if there exists a constant \( C>1 \) such that \( C^{-1}|g_1(z)|\leq |g_2(z)| \leq C|g_1(z)| \) for all \( z \) close to \( z_0 \).}
\begin{equation}
\label{Srho-ai}
\big|S_\rho\big(z^{(0)}\big)\big| \sim \left\{ \begin{array}{rl}
|z-a_i|^{-(2\alpha_i+1)/4} & \text{as} \quad z\to a_i, \medskip \\
|z|^{(-1)^j \mathrm{Re}(\nu)} & \text{as} \quad \mathcal Q_j\ni z\to 0,
\end{array}
\right.
\end{equation}
for \( i,j\in\{1,2,3,4\} \), where \( \mathcal Q_j \) is the \( j \)-th quadrant and \( \nu \) is given by \eqref{nu}.
\end{prop}

Proposition~\ref{prop:szego} is proved in Section~\ref{sec:5}.

\subsection{Theta Function}
\label{ssec:theta}

As it turns out, the product \( (S_\rho\Phi^n)(\z) \) is not sufficient to capture the strong asymptotics of the polynomials \( Q_n(z) \). What needs to be done now is to remove the jumps of this product from the cycles of the homology basis. This is done with the help of the functions \( T_k(\z) \), \( k\in\{0,1\} \), constructed further below in \eqref{T0T1}.

Let \( \mathsf{Jac}(\RS) := \C/ \{\Z+\mathsf B\Z\} \) be the Jacobi variety of \( \RS \), where \( \mathsf B \) is given by \eqref{B}. We shall represent elements of \( \mathsf{Jac}(\RS) \) as equivalence classes \( [s] = \{s+ l + \mathsf Bm:l,m\in \Z\} \), where \( s\in\C \). Since \( \RS \) has genus 1, Abel's map
\[
\z\in\RS \mapsto \left[\int_{\boldsymbol a_3}^\z\Omega\right]\in \mathsf{Jac}(\RS)
\]
is a holomorphic bijection. Hence, given any \( s \in \C \), there exists a unique \( \z_{[s]} \in\RS \) such that \( \left[\int_{\boldsymbol a_3}^{\z_{[s]}}\Omega\right] = [s] \).

Denote by \( \theta(\zeta) \) the Riemann theta function associated to \( \mathsf B \), i.e.,
\[
\theta(\zeta) := \sum_{n \in \Z} \exp \left \{ \pi \ic \mathsf B n^2 + 2\pi \ic n\zeta \right\}.
\]
As shown by Riemann, $\theta(\zeta)$ is an entire, even function that satisfies
\begin{equation}
\label{theta-periods}
\theta(\zeta + l+ m\mathsf B) = \theta(\zeta)\exp\{ -\pi \ic m^2\mathsf B -2\pi \ic m\zeta \}
\end{equation}
for any integers \( l,m \). Moreover, its zeros are simple and \( \theta\left(\zeta \right) = 0 \) if and only if  \( [\zeta] = [(1+\mathsf B)/2] \). The constant \( (1+\mathsf B)/2 \), known as the Riemann constant, will appear often in our computations. So, we choose to abbreviate the representatives of its ``half''-classes by
\begin{equation}
\label{Kpm}
\mathsf K_+ := (1+\mathsf B)/4 \qandq \mathsf K_- := (1-\mathsf B)/4,
\end{equation}
i.e., \([2\mathsf K_+]=[2\mathsf K_-]\). The symmetries of \( \Omega(\z) \) (\( \Omega(-\z)=-\Omega(\z) = \Omega(\z^*) \)) yield that
\begin{equation}
\label{infinite-integral}
\int_{\infty^{(1)}}^{\infty^{(0)}} \Omega = \frac12\int_{\boldsymbol\delta}\Omega = 2\mathsf K_+ \quad \Rightarrow \quad \int_{\boldsymbol a_3}^{\infty^{(k)}} = (-1)^k\mathsf K_+,
\end{equation}
\( k\in\{0,1\} \), where \( \boldsymbol\delta =\pi^{-1}\big((-\infty,-a]\cup[a,\infty)\big) \) is a cycle on \( \RS \) oriented from \( \infty^{(1)} \) to \( \infty^{(0)} \) (on Figure~\ref{f:basis}, \( \boldsymbol\delta \) would be represented by the anti-diagonal), which is clearly homologous to \( \ualpha+\ubeta \). 

With \( c_\rho \) given by \eqref{crho}, define
\begin{equation}
\label{T0T1}
T_k(\z) := \exp \left \{ \pi \ic k\int_{\boldsymbol a_3}^{\z} \Omega \right \}\dfrac{\theta \big( \int_{\boldsymbol a_3}^{\z} \Omega - c_\rho - (-1)^k\mathsf K_+ \big)}{\theta \big( \int_{\boldsymbol a_3}^{\z} \Omega - \mathsf K_+ \big)}
\end{equation}
for \( k\in\{0,1\} \) and \( \z\in\RS_{\ualpha,\ubeta} \), where the path of integration lies entirely within $\RS_{\ualpha, \ubeta}$. Each \( T_k(\z) \) is a meromorphic function that is finite and non-vanishing except for a simple pole at \( \infty^{(1)} \), see \eqref{infinite-integral}, and a simple zero at \( \z_k:=\z_{[c_\rho-(-1)^k\mathsf K_+]} \), where \( \z_k\in\RS \) is uniquely characterized by 
\begin{equation}
\label{lms}
\int_{\boldsymbol a_3}^{\z_k}\Omega = c_\rho -(-1)^k\mathsf K_+ + l_k +m_k\mathsf B,
\end{equation}
\( k\in\{0,1\} \), for some \( l_0,m_0,l_1,m_1\in\Z \). Furthermore, it follows from the normalization in \eqref{Omega}, the definition of \( \mathsf B \) in \eqref{B}, and \eqref{theta-periods} that
\begin{equation}
\label{T-jump}
T_{k+}(\s) = T_{k-}(\s)\left\{
\begin{array}{ll}
\displaystyle \exp\big\{2\pi \ic (k/2- c_{\rho}) \big\}, &  \s\in\ualpha, \medskip \\
\displaystyle \exp\big\{\pi \ic k \big\}, &  \s\in\ubeta.
\end{array}
\right.
\end{equation}


Now we are ready to define the function that will be responsible for the asymptotic behavior of the polynomials \( Q_n(z) \). Given \( \rho(s)\in\mathcal W_1 \), let \( c_\rho \) be defined by \eqref{crho}. Set
\[
\{0,1\}\ni \imath(n) := n\mod 2, \quad n\in\Z,
\]
to be the parity function. Then it follows from \eqref{Phi-jump}, \eqref{S-jump}, and \eqref{T-jump} that the function
\begin{equation}
\label{Psin}
\Psi_n(\z) := \big(\Phi^nS_\rho T_{\imath(n)} \big)(\z), \quad \z\in\RS\setminus\bd,
\end{equation}
is meromorphic in \( \RS\setminus\bd \) with a pole of order \( n \) at \( \infty^{(0)} \), a zero of multiplicity \( n-1 \) at \( \infty^{(1)} \), a simple zero at \( \z_{\imath(n)} \), and otherwise non-vanishing and finite, whose traces on \( \bd  \) satisfy
\begin{equation}
\label{Psin-jump}
\Psi_{n+}(\s) =\Psi_{n-}(\s)/(\rho w_+)(s), \quad \s\in\bd,
\end{equation}
and whose behavior around the ramification points of \( \RS \) as well as  \( {\boldsymbol 0^*},{\boldsymbol 0} \) is governed by \eqref{Srho-ai}. 

\subsection{Asymptotics}
\label{ssec:asymp}

In this section we formulate the main theorem on the behavior of the polynomials \( Q_n(z) \). As was alluded to in the introduction, we do not expect to be able to handle all the possible indices \( n \) as \( Q_n(s) \) might have degree smaller than \( n \). One source of this degeneration already can be seen from \eqref{Psin} since this function can have a pole of order \( n-1 \) at \( \infty^{(0)} \) when \( \z_{\imath(n)}=\infty^{(0)} \). In fact, this is the only reason for the degeneration in the generic cases described in \cite{ApY15}. However, this is no longer the case for the considered model.

To restrict the indices we need the following, unfortunately very technical, definition. Let us set
\begin{equation}
\label{sigma-o}
\varsigma_\nu := \left\{ \begin{array}{rl} 1, & \re(\nu)>0, \medskip \\-1, & \re(\nu)<0, \end{array}\right. \qandq \boldsymbol o := \left\{ \begin{array}{rl} \boldsymbol 0, & \re(\nu)>0, \medskip \\ \boldsymbol 0^*, & \re(\nu)<0. \end{array}\right.
\end{equation}
We do not make any choice for \( \varsigma_\nu \) and \( \boldsymbol o \) when \( \re(\nu)=0 \). Given \( \rho(s)\in\mathcal W_1 \) and the constant \( c_\rho \) from \eqref{crho}, define
\begin{equation}
\label{Arhon}
A_{\rho,n} := \left\{ \begin{array}{rl} \sigma_{\imath(n)}A_{\rho,n}^\prime\Phi(\z_{\imath(n)})\Phi^{2(n-1)}(\boldsymbol o), & \re(\nu)\neq 0, \medskip \\
0, & \re(\nu)=0, \end{array}\right.
\end{equation}
where \( \sigma_k := (-1)^{l_k+m_k+ k} \), \( k\in\{0,1\} \), see \eqref{lms}, and
\begin{multline*}
A_{\rho,n}^\prime := A_\rho e^{\pi\ic\varsigma_\nu( c_\rho+1/4)} \frac{\sqrt{a^2+b^2}}2\frac{\Gamma(1-\varsigma_\nu\nu)}{\sqrt{2\pi}}  \times \\  \left[\lim_{z\to0,\arg(z)=5\pi/4}|z|^{2\nu}S^2_\rho\big(z^{(0)}\big)\right]^{\varsigma_\nu}\left(\frac{ab}{2n}\right)^{1/2-\varsigma_\nu\nu},
\end{multline*}
and
\[
A_\rho := e^{\pi\ic\nu}\rho_3(0)\frac{(\rho_2+\rho_3)(0)}{\rho_2(0)} \quad \text{or} \quad A_\rho := \frac1{(ab)^2}\frac{(\rho_3+\rho_4)(0)}{(\rho_3\rho_4)(0)}
\]
depending on whether \( \re(\nu)>0 \) or \( \re(\nu)<0 \) (it follows from the last display in Section~\ref{sec:5}, devoted to the proof of Proposition~\ref{prop:szego}, that the limit in the definition of the constant \( A_{\rho,n}^\prime \) is indeed well defined). 

Given the above constants \( A_{\rho,n} \) and  \( \varepsilon\in(0,1/2) \), we define subsequences of allowable indices \( n \) for the weight \( \rho(s) \) by
\begin{equation}
\label{Nrhoepsilon}
\N_{\rho,\varepsilon} := \left\{n\in\N: \z_{\imath(n)}\neq\infty^{(0)}\;\text{and}\;\;|1-A_{\rho,n}|\geq\varepsilon\right\}.
\end{equation}
The following proposition states that such sequences are non-empty.

\begin{prop}
\label{prop:N}
Let \( \N_{\rho,\varepsilon} \) be given by \eqref{Nrhoepsilon}. If \( [c_\rho] =[0] \) or \( [c_\rho] =[(1+\mathsf B)/2] \), then it holds that
\begin{equation}
\label{N1}
\N_{\rho,\varepsilon} = \N_\rho := \left\{
\begin{array}{rll}
2\N & \text{when} & [c_\rho] =[0], \medskip \\
\N\setminus2\N & \text{when} & [c_\rho] = [(1+\mathsf B)/2].
\end{array}
\right.
\end{equation}
If \( [c_\rho] \neq [0] \) and \( [c_\rho] \neq  [(1+\mathsf B)/2] \) while \( \re(\nu)\in(-1/2,1/2) \), it holds that
\begin{equation}
\label{N2}
\N_{\rho,\varepsilon} =\N_\rho := \N. 
\end{equation}
If \( [c_\rho] \neq [0] \) and \( [c_\rho] \neq  [(1+\mathsf B)/2] \), and  \( \re(\nu)=1/2 \), then \( \N_{\rho,\varepsilon} \) is an infinite subsequence with gaps of size at most \( 2 \) (clearly, this is the only case when \( \N_{\rho,\varepsilon} \) might depend on \( \varepsilon \)).
\end{prop}
\begin{proof}
It readily follows from \eqref{lms} and \eqref{infinite-integral}  that
\[
[c_\rho] = [k(1+\mathsf B)/2] \quad \Leftrightarrow \quad \z_1 = \infty^{(k)} \quad \Leftrightarrow \quad \z_0=\infty^{(1-k)}
\]
for \( k\in\{0,1\} \) (in which case \( \Phi(\z_{\imath(n)})=\Phi\big(\infty^{(1)}\big)=0=A_{\rho,n}\)). On the other hand, because Abel's map is a bijection, we also get that $|\pi(\z_1)| < \infty \Leftrightarrow |\pi(\z_0)| < \infty$. This proves \eqref{N1}. Observe that
\begin{equation}
\label{Arhon-0}
A_{\rho,n} = B_{\rho,\imath(n)}\Phi(\boldsymbol o)^{2(n-1)}n^{\varsigma_\nu\nu-1/2},
\end{equation}
where \( B_{\rho,\imath(n)} \) depends only on the parity of \( n \) and \( |\Phi(\boldsymbol o)|=1 \) by \eqref{K0-formula}. Hence,  \( A_{\rho,n}\to 0 \) as \( n\to\infty \) when  \( \re(\nu)\in(-1/2,1/2) \), which proves \eqref{N2}. In the remaining situation,
\[
A_{\rho,n} = B_{\rho,\imath(n)}\exp\big\{2(n-1)\ic\arctan(a/b)+\ic\im(\nu)\log n\big\}
\]
by \eqref{K0-formula}. If \( |B_{\imath(n)}|\neq 1 \), then, in fact, \( \N_{\rho,\varepsilon} = \N \). Otherwise, we have that
\[
A_{\rho,n+2}/A_{\rho,n} = \exp\big\{2\ic\arctan(a/b)+\ic\im(\nu)\log(1+ 2/n)\big\}.
\]
As \( \arctan(a/b)\in(0,\pi/2) \) and \( \log(1+2/n)=o(1) \), both constants \( A_{\rho,n+2} \) and \( A_{\rho,n} \) cannot be simultaneously close to 1.
\end{proof}

When \( \re(\nu)<1/2 \), the sequence \( \N_{\rho,\epsilon} = \N_\rho \) is equal to the whole set of the natural numbers or consists of every other one. This is consistent with the explanation given at the beginning of the subsection and is supported by the examples in Sections~\ref{sec:3.1} and~\ref{sec:3.2} where two weights \( \rho(s) \) are provided for which \( Q_{2n}(z)=Q_{2n+1}(z) \). As mentioned before, this is a generic behavior observed in \cite{ApY15}. On the technical level this degeneration manifests itself as our inability to construct the ``global parametrix'', see Section~\ref{sec:6.3}, since we are no longer able to properly renormalize \( Q_n(z) \) by \( \Psi_n(z^{(0)}) \) when \( \z_{\imath(n)}=\infty^{(0)} \).

When \( \re(\nu)=1/2 \), new phenomenon occurs. The sequence \( \N_{\rho,\epsilon} \) can have gaps of size 2 depending on the behavior of the constants \( A_{n,\rho} \). This suggests that there might be indices \( n \) such that \( Q_n(z) = Q_{n+1}(z) = Q_{n+2}(z) \). Such a possibility can in fact occur, see Section~\ref{sec:3.3} for an example. On the technical level, the second condition in \eqref{Nrhoepsilon} appears in an attempt to match the behavior of \( Q_n(z) \) at the origin, that is, during the construction of the so-called ``local parametrix'', see Sections~\ref{sec:4.5} and~\ref{sec:5.2}, and manifests itself through the constants \( L_{ni} \), see \eqref{Lni}.

Recall that the weight \( \rho(s) \) defines two constants: \( \ell \), which says how well the restrictions of \( \rho(s) \) to different segments \( \Delta_i \) match each other at the origin, and \( \nu \), defined in \eqref{nu}. Our analysis does not allow us to handle all possible combinations of these constants. In what follows we assume that
\begin{equation}
\label{ell-nu}
|\re(\nu)| \in \left\{ 
\begin{array}{rcl}
~[0,\sqrt7/2-1)& \text{when} & \ell=1, \medskip \\
~[0,1/2) & \text{when} & \ell=2, \medskip \\
~[0,1/2] & \text{when} & \ell>3.
\end{array}
\right.
\end{equation}
This technical condition appears in the rate of decay of the error, which we quantify by the following exponent:
\begin{equation}
\label{dnuell}
d_{\nu,\ell} := \left\{ \begin{array}{rl}
\frac{(\frac12+|\re(\nu)|)(\ell-2|\re(\nu)|)}{\ell+1+2|\re(\nu)|}, & \ell\geq\frac{4|\re(\nu)|(1+|\re(\nu)|)}{1-2|\re(\nu)|}, \medskip \\
\frac{\ell(3-2|\re(\nu)|)-2|\re(\nu)|(3+2|\re(\nu)|)}{2(\ell+3+2|\re(\nu)|)}, & \text{otherwise},
\end{array}
\right.
\end{equation}
where we understand that \( d_{\nu,\infty}=1/2+|\re(\nu)|\). It is a straightforward computation to check that requiring positivity of the numerator of \( d_{\nu,\ell} \) in the second line of \eqref{dnuell} produces restriction \eqref{ell-nu}.  Observe also that \( d_{1/2,\ell} = \frac{\ell-2}{\ell+4} \). 

\begin{thm}
\label{thm:asymptotics}
Let \( \rho(s)\in \mathcal W_\ell \), where \( \ell \) is a positive integer or infinity. Define \( \nu \) by \eqref{nu} and assume that \eqref{ell-nu} is satisfied. Let \( \Psi_n(\z) \) be given by \eqref{Psin} and \( \N_{\rho,\varepsilon} \) be as in \eqref{Nrhoepsilon} for some  \(\varepsilon\in(0,1/2) \) fixed. Then it holds for all \( n\in\N_{\rho,\varepsilon} \) large enough that
\begin{equation}
\label{Qnasymp}
Q_n(z) = \gamma_n \big(1 + \upsilon_{n1}(z) \big)\Psi_n\big(z^{(0)} \big) +\gamma_n \upsilon_{n2}(z) \Psi_{n-1}\big( z^{(0)} \big)
\end{equation}
for $z\in\C\setminus\Delta$, where \( \gamma_n:=\lim_{z\to\infty}z^n\Psi_n^{-1}\big(z^{(0)}\big) \) is the normalizing constant;
\begin{multline}
\label{Qnasymp1}
Q_n(s) = \gamma_n\big(1 + \upsilon_{n1}(s) \big)\left(\Psi_{n+}^{(0)}(s) + \Psi_{n-}^{(0)}(s) \right) + \\ \gamma_n \upsilon_{n2}(s)\left(\Psi_{n-1+}^{(0)}(s) + \Psi_{n-1-}^{(0)}(s) \right)
\end{multline}
for \( s\in\Delta^\circ \), where \( \Psi_{n\pm}^{(0)}(s) \) are the traces of \( \Psi_n\big(z^{(0)} \big) \) on the positive and negative sides of \( \Delta \). The functions \( \upsilon_{ni}(z) \) are such that
\begin{equation}
\label{upsilons}
\upsilon_{ni}(\infty)=0 \qandq  \upsilon_{ni}(z) = L_{n,i}z^{-1} + \mathcal O\left(n^{-d_{\nu,\ell}}\right)
\end{equation}
where \( \mathcal O(\cdot) \) holds locally uniformly on \( \overline\C\setminus\Delta \) in \eqref{Qnasymp} and on \( \Delta^\circ \) in \eqref{Qnasymp1}, \( d_{\nu,\ell} \) was defined in \eqref{dnuell}, and \( L_{ni} \) are constants given by
\begin{equation}
\label{Lni}
L_{ni} = (-1)^{\imath(n)}\frac{A_{\rho,n}}{1-A_{\rho,n}}\left(-\frac{\Phi T_{\imath(n)}}{T_{\imath(n-1)}} \right)^{i-1}(\boldsymbol o)\frac{(T_0/T_1)(\boldsymbol o)}{(T_0/T_1)^\prime(\boldsymbol o)}
\end{equation}
when \( |\pi(\z_k)|<\infty \), \( i\in\{1,2\} \), where  \( \boldsymbol o \) was defined in \eqref{sigma-o} (when \( |\pi(\z_k)|=\infty \), the expressions for \( L_{ni} \) are even more cumbersome and therefore are omitted here).

\end{thm}

Notice that the behavior of the polynomials \( Q_n(z) \) is qualitatively different for \( \re(\nu)<1/2 \) and \( \re(\nu)=1/2 \) as the first summand in \eqref{upsilons} is decaying in the former case by \eqref{Arhon-0}, but does not decay in the latter.

Recall that the traces of \( \Phi(\z) \) are unimodular on \( \bd \), see \eqref{Green}. Since \( \Psi_n(\z) = (S_\rho T_{\imath(n)})(\z)\Phi^n(\z) \), it is exactly the sum of the terms \( \big(\Phi_+^{(0)}(s)\big)^n \) and \( \big(\Phi_-^{(0)}(s)\big)^n \) that creates oscillations describing the zeros of \( Q_n(z) \). Of course, since the traces of \( (S_\rho T_{\imath(n)})^{(0)}_{\pm}(s) \) are in general complex-valued, the zeros of \( Q_n(z) \) do not lie exactly on \( \Delta \). However, we do prove that \eqref{Qnasymp1} holds on compact subsets ``close'' to \( \Delta^\circ \), where \( \Psi_{n\pm}^{(0)}(s) \) are analytically continued from \( \Delta^\circ \) into the complex plane with the help of \eqref{Psin-jump}.

When \( \ell<\infty \), we cannot control the error functions \( \upsilon_{ni}(z) \) around the origin and therefore cannot describe the polynomials \( Q_n(z) \) there (however, we can extend \eqref{Qnasymp1} to hold on a sequence of compact subsets of \( \Delta^\circ \) that are allowed to approach the origin with a certain speed at the expense of worsening the rate of decay in the error estimates). When \( \ell=\infty \), we can provide an asymptotic formula for \( Q_n(z) \) around the origin, but due to its technical nature we placed it at the very end of the paper in Section~\ref{ssec:at0}.

Theorem~\ref{thm:asymptotics}, as well as Theorem~\ref{thm:pade} further below, is proved in Section~\ref{sec:6} with the derivation of some technical identities relegated to Section~\ref{sec:4}.

\subsection{Pad\'e Approximation}
\label{ssec:pade}

For an integrable weight \( \rho(s) \) on \( \Delta \) define
\begin{equation}
\label{hatrho}
\widehat\rho(z) := \frac1{2\pi\ic}\int_\Delta\frac{\rho(s)\dd s}{s-z}, \quad z\in\overline\C\setminus\Delta.
\end{equation}
In particular, it can be readily verified that the functions \( \sum_{i=1}^4 C_i\log(z-a_i) \) and \( \prod_{i=1}^4(z-a_i)^{\alpha_i} \), where the constants \( C_i \) add up to zero and the exponents \( -1<\alpha_i\not\in\Z \) add up to an integer, possess branches holomorphic off \( \Delta \) that can be represented by \eqref{hatrho} for certain weight functions in \( \mathcal W_\infty \) (the second function can be represented by \eqref{hatrho} up to an addition of a polynomial).

Given \( \widehat\rho(z) \), it follows from orthogonality relations \eqref{orthogonality} that
\begin{equation}
\label{linear-system}
R_n(z) := \displaystyle \frac1{2\pi\ic}\int_\Delta\frac{Q_n(s)\rho(s)\dd s}{s-z} = \mathcal O\big(z^{-n-1} \big) \qasq z\to\infty.
\end{equation}
Observe also that \( R_n(z) \) can be rewritten as
\[
R_n(z) = \big(Q_n\widehat\rho\big)(z) +  \frac1{2\pi\ic}\int_\Delta\frac{Q_n(s)-Q_n(z)}{s-z}\rho(s)\dd s = \big(Q_n\widehat\rho\big)(z) - P_n(z),
\]
where \( P_n(z) \) is a polynomial of degree at most \( n-1 \). The rational function \( (P_n/Q_n)(z) \) is called the \( n \)-th diagonal Pad\'e approximant of \( \widehat\rho(z) \). 

\begin{thm}
\label{thm:pade}
Let \( \widehat\rho(z) \) be given by \eqref{hatrho} and \( R_n(z) \) be defined by \eqref{linear-system}. In the setting of Theorem~\ref{thm:asymptotics},  it holds for all \( n\in\N_{\rho,\varepsilon} \) large enough that
\begin{equation}
\label{Rnasymp}
(wR_n)(z) = \gamma_n \big(1 + \upsilon_{n1}(z) \big) \Psi_n\big(z^{(1)}\big) +\gamma_n \upsilon_{n2}(z) \Psi_{n-1}\big(z^{(1)}\big)
\end{equation}
locally uniformly in $\overline\C\setminus\Delta$, where \( \upsilon_{ni}(z) \) are the same as in Theorem~\ref{thm:asymptotics}.
\end{thm}

Theorem~\ref{thm:pade} has the following consequences for Pad\'e approximation: it holds that
\begin{multline*}
\widehat\rho(z) - \frac{P_n(z)}{Q_n(z)} = \frac{R_n(z)}{Q_n(z)} = \frac1{w(z)(S_\rho^2\Phi^{2n-1})(z^{(0)})} \times \\ \frac{\big(1 + \upsilon_{n1}(z) \big)(\Phi T_{\imath(n)})(z^{(1)}) + \upsilon_{n2}(z)T_{\imath(n-1)}(z^{(1)})}{\big(1 + \upsilon_{n1}(z) \big) T_{\imath(n)}(z^{(0)}) + \upsilon_{n2}(z)(T_{\imath(n-1)}/\Phi)(z^{(0)})},
\end{multline*}
where we used \eqref{Psin} and the fact that \( S_\rho(\z)S_\rho(\z^*) \equiv 1 \) and \( \Phi(\z)\Phi(\z^*)\equiv 1 \). It follows from \eqref{Green} that the first fraction on the right-hand side of the equality above is geometrically small in \( \overline\C\setminus\Delta \) with the zero of order \( 2n+1 \) at infinity. However, if \( \z_{\imath(n)} \in \RS^{(0)}\setminus\bd \), the second fraction, and hence the Pad\'e approximant, will have a pole in the vicinity of  \( z_{\imath(n)} \) by Rouche's theorem, which will prevent the convergence around \( z_{\imath(n)} \). On the other hand, if  \( \z_{\imath(n)} \in \RS^{(1)}\setminus\bd \), then Rouche's theorem yields that the Pad\'e approximant has an additional interpolation point near \( \z_{\imath(n)} \).

\section{Examples}

In this section, we illustrate Theorem~\ref{thm:asymptotics} by three examples. In them, we shall not compute \( S_\rho(\z) \) and \( c_\rho \) via their integral representations, \eqref{szego} and \eqref{crho}, but rather construct a candidate \( \widehat S_\rho(\z) \) with the desired jump over \( \boldsymbol\Delta \) and the singular behavior as in \eqref{Srho-ai}. This construction will also determine a candidate constant \(  \widehat c_\rho \).  It is simple to argue that
\[
S_\rho(\z) = \widehat S_\rho(\z)\exp\left\{2\pi\ic m\int_{\boldsymbol a_3}^\z\Omega\right\}, \quad c_\rho=\widehat c_\rho -m\mathsf B,
\]
for some integer \( m \). Using \( \widehat c_\rho \) in \eqref{T0T1}, we then construct \( \widehat T_{\imath(n)}(\z) \) for which it holds that
\[
T_{\imath(n)}(\z) = \widehat T_{\imath(n)}(\z)\exp\left\{-2\pi\ic m\int_{\boldsymbol a_3}^\z\Omega-\pi\ic m^2\mathsf B + 2\pi\ic(-1)^{\imath(n)}\mathsf K_+\right\}
\]
with the same integer \( m \). This means that
\[
\big(S_\rho T_{\imath(n)}\big)(\z)/\big(S_\rho T_{\imath(n)}\big)\big(\infty^{(0)}\big) = \big(\widehat S_\rho \widehat T_{\imath(n)}\big)(\z)/\big(\widehat S_\rho \widehat T_{\imath(n)}\big)\big(\infty^{(0)}\big)
\]
and therefore \eqref{Qnasymp} and \eqref{Rnasymp} remain valid with \( S_\rho(\z) \), \( T_{\imath(n)}(\z) \) replaced by \( \widehat S_\rho(\z) \), \( \widehat  T_{\imath(n)}(\z) \). Furthermore, the value of \( A_{\rho,n} \) in \eqref{Arhon} will not change either as the limit in the definition of \( A_{\rho,n}^\prime \) will be augmented by \( e^{\pi\ic m(1-\mathsf B)} \), see \eqref{0-integral}, that will be offset by the change in \( c_\rho \) and \( \sigma_k \) (\( \widehat\sigma_k = (-1)^m\sigma_k \)). Thus, with a slight abuse of notation, we shall keep on writing \( S_\rho(\z) \), \( T_{\imath(n)}(\z) \) below.

\subsection{Chebysh\"ev-type case} 
\label{sec:3.1}

Let \(2\widehat\rho(z)=1/w(z) \), in which case it holds that
\[
\rho(s)=1/w_+(s), \quad s\in\Delta,
\]
where \( \widehat\rho(z) \) and \( w(z) \) were defined in \eqref{hatrho} and \eqref{w}, respectively, and the implication follows from the Plemelj-Sokhotski formulae and Privalov's theorem. Using analytic continuations of \( w(z) \) one can easily see that \( \rho(s)\in\mathcal W_\infty \) and \( \nu=0 \).  Since \( (\rho w_+)(s)\equiv1 \), we get that \( S_\rho(\z)\equiv 1 \) and necessarily \( c_\rho=0 \). Thus, \( \N_{\rho,\varepsilon} =2\N \) and \( \z_0=\infty^{(1)} \) (\( \z_1=\infty^{(0)}\)). Moreover, we get that \( T_0(\z)\equiv1 \) and \( T_1(\z) = 1/\Phi(\z) \), see \eqref{Phitheta}. Hence, it follows from \eqref{PhiSquared} and \eqref{Qnasymp} that
\[
Q_{2n}(z) = \frac{1+o(1)}{2^n}\left(z^2 + \frac{b^2-a^2}2 + w(z)\right)^n,
\]
where it holds that \( o(1) \) is geometrically small on closed subsets of \( \overline\C\setminus\Delta \) (see \cite{Y15} for the error rate in this case). To show that the above result is in a way best possible, assume that \( a=b=1 \). Recall that the \( n \)-th monic Chebysh\"ev polynomial of the first kind is defined by
\[
2^nT_n(z) = \left(z+\sqrt{z^2-1}\right)^n + \left(z-\sqrt{z^2-1}\right)^n
\]
and is orthogonal to \( x^j \), \( j\in\{0,\ldots,n-1\} \), on \( (-1,1) \) with respect to the weight \( 1/\sqrt{1-x^2} \). Hence,
\begin{multline*}
\ic\int_\Delta s^kT_n\big(s^2\big)\rho(s)\dd s  = \\ \left(\int_0^1-\int_{-1}^0\right)\frac{x^kT_n\big(x^2\big)\dd x}{\sqrt{1-x^4}} - \ic^{k+1}\left(\int_0^1-\int_{-1}^0\right)\frac{x^kT_n\big(-x^2\big)\dd x}{\sqrt{1-x^4}}.
\end{multline*}
Clearly, the above expression is zero for all even \( k \). Assume now that \( k=2j+1 \), \( j\in\{0,\ldots,n-1\} \). Then we can continue the above chain of equalities by
\[
\int_0^1 \frac{x^jT_n(x)\dd x}{\sqrt{1-x^2}} - (-1)^{j+1}\int_0^1\frac{x^jT_n(-x)\dd x}{\sqrt{1-x^2}} = \int_{-1}^1\frac{x^jT_n(x)\dd x}{\sqrt{1-x^2}} = 0,
\]
where the last equality follows from the orthogonality properties of the Chebysh\"ev polynomials. Thus, it holds that
\[
Q_{2n+1}(z) = Q_{2n}(z) = T_n\big(z^2\big)
\]
in this case, which justifies the exclusion of odd indices from \( \N_\rho=\N_{\rho,\varepsilon} \) as for such indices polynomials can and do degenerate.

\subsection{Legendre-type case}
\label{sec:3.2}

Let \( \widehat\rho(z) = \frac1{2\pi\ic}\big(\log(z^2-1)-\log(z^2+1)\big) \), in which case it holds that
\[
\rho(s)=(-1)^i, \quad s\in\Delta_i,
\]
\( i\in\{1,2,3,4\} \), where the justification for the implication is the same as before. As in the previous case, it holds that \( \nu=0 \). Let \( \sqrt w(z) \) be the branch holomorphic in \( \C\setminus\Delta \) such that \( \sqrt w(z) = z + \mathcal O(1) \) as \( z\to\infty \).  Further, let
\[
\Phi_*(z) := \sqrt{\frac2{a^2+b^2}}\left(z^2 + \frac{b^2-a^2}2 + w(z)\right)^{1/2},
\]
be the branch holomorphic in \( \C\setminus\Delta \) such that \( \Phi_*(z) = z + \mathcal O(1) \) as \( z\to\infty \). It easily follows from \eqref{Phi-jump}, \eqref{PhiSquared}, and \eqref{log-cap} that \( \Phi_*(z) \) is an analytic continuation of \( -\Phi\big(z^{(0)}\big) \) across \( \pi(\boldsymbol\alpha) \cup \pi(\boldsymbol\beta) \). It is now straightforward to check that
\[
S_\rho\big( z^{(0)} \big) = e^{-\pi\ic/4}\Phi_*(z)/\sqrt w(z)
\]
and thus \( c_\rho = 0 \). Hence, as in the previous subsection, \( \N_{\rho,\varepsilon}=2\N \) and \( T_0(\z)\equiv1 \) while \( T_1(\z) = 1/\Phi(\z) \). Therefore, we again deduce from \eqref{PhiSquared} and \eqref{Qnasymp} that
\[
Q_{2n}(z) = \frac{1+\mathcal O(n^{-1/2})}{2^{n+1/2}\sqrt w(z)}\left(z^2 + \frac{b^2-a^2}2 + w(z)\right)^{n+1/2},
\]
uniformly on closed subsets of \( \overline\C\setminus\Delta \). Again, to show that the above result is best possible, assume that \( a=b=1 \). Then we can check exactly as in the previous subsection that 
\[
Q_{2n+1}(z)=Q_{2n}(z) = L_n\big(z^2\big),
\]
where \( L_n(x) \) is the \( n \)-th monic Legendre polynomial, that is, degree \( n \) polynomial orthogonal to \( x^j \), \( j\in\{0,\ldots,n-1\} \), on \( (-1,1) \) with respect to a constant weight.

\subsection{Jacobi-\(1/4\) case}
\label{sec:3.3}

Let \( \sqrt2\widehat\rho(z)=1/\sqrt w(z) \), in which case it holds that
\[
\rho(s)=-\ic^{4-i}/|\sqrt w(s)|, \quad s\in\Delta_i, \quad  i\in\{1,2,3,4\},
\]
where \( \sqrt w(z) \) is the branch defined in the previous subsection. Observe that
\[
(\rho w_+)(s)=\ic^{i-1}|\sqrt w(s)|, \quad s\in\Delta_i,
\]
and that \( \nu=1/2 \). In particular, the constant \( A_\rho \) appearing in the definition of \( A_{\rho,n} \) in \eqref{Arhon} is equal to \( A_\rho = \sqrt2e^{-\pi\ic/4}/\sqrt{ab} \).

To construct a Szeg\H{o} function of \( \rho(s) \), let
\[
\Theta^2(\z) := \frac{\theta\big(\int_{\boldsymbol a_3}^\z\Omega + \mathsf K_-\big)}{\theta\big(\int_{\boldsymbol a_3}^\z\Omega - \mathsf K_-\big)} \frac{\theta\big(\int_{\boldsymbol a_3}^\z\Omega - \mathsf K_+\big)}{\theta\big(\int_{\boldsymbol a_3}^\z\Omega + \mathsf K_+\big)}, \quad z\in\RS_{\ualpha,\ubeta},
\]
where the path of integration lies entirely in \( \RS_{\ualpha,\ubeta} \). It follows from \eqref{infinite-integral} and \eqref{0-integral} further below that \( \Theta^2(\z) \) is a meromorphic function in \( \RS_{\ualpha,\ubeta} \) with two simple poles, namely, \( \infty^{(0)}, {\boldsymbol 0} \), and two simple zeros \( \infty^{(1)},{\boldsymbol 0^*} \). Moreover, \( \Theta^2(\z) \) is continuous across \( \ubeta \) and satisfies \( \Theta^2_+(\s)=\Theta^2_-(\s)e^{-2\pi\ic \mathsf B} \) on \( \ualpha \) by \eqref{theta-periods} and \( \Theta^2(\z)\Theta^2(\z^*) \equiv 1 \) by the symmetries of \( \theta(\zeta) \) and \( \Omega(\z) \). Since each individual fraction in the definition of \( \Theta^2(\z) \) is injective, we can define a branch \( \Theta(\z) \) such that 
\[
\Theta_+(\s) = \Theta_-(\s)\left\{
\begin{array}{ll}
e^{-\pi\ic\mathsf B}, & \s \in \ualpha, \medskip \\
-1,& \s\in \boldsymbol\Delta_3\cup\pi^{-1}((-\infty,-a]),
\end{array}
\right.
\]
and \( \Theta(\z)\Theta(\z^*) \equiv 1 \). Further, let \( w^{1/4}(z) \) be the branch holomorphic in \( \C\setminus\big(\Delta\cup(-\infty,a)\big) \) that is positive for \( z>a \). Now, one can verify that \( c_\rho=-\mathsf B/2 \) and
\[
S_\rho\big(z^{(k)} \big) = \Theta\big(z^{(k)}\big)w^{\frac{2k-1}4}(z), \quad k\in\{0,1\}.
\]

Let us now compute \( A_{\rho,n}^\prime \) appearing in \eqref{Arhon}. Since \( \sqrt w(z)\to e^{-3\pi\ic/4}\sqrt{ab} \) as  \(\mathcal Q_3\ni z\to 0 \), we get that
\begin{eqnarray*}
\lim_{z\to0,\arg(z)=5\pi/4}|z|S^2_\rho\big(z^{(0)}\big) &=& \frac{e^{-\pi\ic/2}}{\sqrt{ab}}\lim_{\mathcal Q_3\ni z\to0}z\Theta^2\big(z^{(0)}\big) \\
&=&  e^{\pi\ic\mathsf B/2}\frac{2\sqrt{ab}}{\sqrt{a^2+b^2}}\Phi({\boldsymbol 0}),
\end{eqnarray*}
where the second equality follows from \eqref{0-integral}, \eqref{Phi01}, \eqref{z-theta}, and \eqref{moduli} further below. Therefore, it holds that \( A_{\rho,n}^\prime = \Phi(\boldsymbol 0)\). It is easy to see from \eqref{0-integral} that \( \z_0 = {\boldsymbol 0} \), \( l_0=0,m_0=1 \), and \( \z_1={\boldsymbol 0^*} \), \( l_1=m_1=0 \). Therefore, \( \sigma_{\imath(n)}=-1 \) and the condition defining \( \N_{\rho,\varepsilon} \) in Proposition~\ref{prop:N} specializes to
\[
\left|1 + \exp\big\{ 2\ic(n-\imath(n))\arctan(a/b)\big\}\right|>\varepsilon
\] 
by \eqref{K0-formula} and since \( \Phi(\z_1)\Phi(\z_0)=1 \), see \eqref{ProdPhizk} further below. As \( T_0(\boldsymbol 0) =0 \) and respectively \( L_{n1}=0 \), we then get that \( Q_n(z) \), \( n\in\N_{\rho,\varepsilon} \), is equal to
\[
\gamma_n\big(S_\rho\Phi^n\big)\big(z^{(0)}\big)\left\{
\begin{array}{ll}
\big(T_0\big(z^{(0)}\big)+\mathcal O_\varepsilon(n^{-1})\big), & n\in2\N, \medskip \\
\big(T_1\big(z^{(0)}\big)+z^{-1}L_{n2}(T_0/\Phi)\big(z^{(0)}\big)+\mathcal O_\varepsilon(n^{-1})\big), & n\not\in2\N,
\end{array}
\right.
\]
uniformly on closed subsets of \( \overline\C\setminus\Delta \), where
\[
L_{n2} = \frac{-1}{(T_0/T_1)^\prime(\boldsymbol 0)}\frac{\Phi^{2n-1}(\boldsymbol 0)}{1+\Phi^{2(n-1)}(\boldsymbol 0)}
\]
for all odd \( n \).  When \( a=b \), we further get that \( L_{n2} = -e^{\pi\ic/4}/[2(T_0/T_1)^\prime(\boldsymbol 0)] \) for \( n\in\N_{\rho,\varepsilon} \) and
\[
\N_{\rho,\varepsilon} = \{ n=4k,4k+1:~k\in\N \}.
\]
Assume further that \( a=b=1 \) and let \( P_{n,1}(x) \) be the \( n \)-th degree monic polynomial orthogonal on \( [0,1] \) to \( x^j \), \( j\in\{0,\ldots,n-1 \} \), with respect to the weight function \( x^{-3/4}(1-x)^{-1/4} \). Then
\[
\int_\Delta s^kP_{n,1}\big(s^4\big)\rho(s)\dd s = \big(1+\ic^k\big)\int_{-1}^1 y^kP_{n,1}\big(y^4\big)\frac{\dd y}{(1-y^4)^{1/4}},
\]
which is equal to zero for all \( k \) odd by symmetry and for all \( k=4j+2 \) due to the factor \( 1+\ic^k \). When \( k=4j \), \( j\in\{0,\ldots,n-1\} \), we can further continue the above equality by
\[
4\int_0^1 y^{4j}P_{n,1}\big(y^4\big)\frac{\dd y}{(1-y^4)^{1/4}} = \int_0^1 x^jP_{n,1}(x)\frac{\dd x}{x^{3/4}(1-x)^{1/4}} = 0,
\]
where the last equality now holds by the very choice of \( P_{n,1}(z) \). Hence, it holds that
\[
Q_{4n}(z) = P_{n,1}\big(z^4\big) \qandq Q_{4n+1}(z) = Q_{4n+2}(z) = Q_{4n+3}(z) = zP_{n,2}\big(z^4\big),
\]
where the second set of relations can be shown similarly with \( P_{n,2}(x) \) being the \( n \)-th degree monic polynomial orthogonal on \( [0,1] \) to \( x^j \), \( j\in\{0,\ldots,n-1 \} \), with respect to the weight function \( x^{1/4}(1-x)^{-1/4} \). That is, the restriction to the sequence of indices \( \{ n=4k,4k+1:~k\in\N \} \) is not superfluous and the main term of the asymptotics of the polynomials does depend on the parity of \( n \).

\section{Auxiliary Identities}
\label{sec:4}

In this section we state a number of identities, some of which we have already used and some of which we shall use later.

\begin{lem}
Recall \eqref{Kpm}. It holds that
\begin{equation}
\label{0-integral}
\int_{\boldsymbol a_3}^{\boldsymbol 0}\Omega = -\mathsf K_- \qandq \int_{\boldsymbol a_3}^{\boldsymbol 0^*}\Omega = \mathsf K_-,
\end{equation}
where the path of integration lies entirely in \( \RS_{\ualpha,\ubeta} \).
\end{lem}
\begin{proof}
Exactly as in the case of \eqref{infinite-integral}, the symmetries of \( \Omega(\z) \) imply that
\[
-\int_{\boldsymbol a_3}^{\boldsymbol 0}\Omega = \int_{\boldsymbol a_3}^{\boldsymbol 0^*}\Omega = \frac12\int_{\boldsymbol \Delta_3} \Omega = \frac14\int_{\boldsymbol \Delta_3-\boldsymbol \Delta_1} \Omega. 
\]
The claim now follows from the fact that \( \boldsymbol\Delta_3 - \boldsymbol \Delta_1 \) is homologous to \( \ualpha-\ubeta \).
\end{proof}

\begin{lem}
It holds that
\begin{equation}
\label{Phitheta}
\Phi(\z) = \exp \left \{ -\pi \ic \int_{\boldsymbol a_3}^{\z} \Omega \right \} \dfrac{\theta \big( \int_{\boldsymbol a_3}^{\z} \Omega - \mathsf K_+ \big)}{\theta \big( \int_{\boldsymbol a_3}^{\z} \Omega + \mathsf K_+ \big)}.
\end{equation}
\end{lem}
\begin{proof}
It follows from \eqref{infinite-integral} and \eqref{theta-periods} that the right hand side of \eqref{Phitheta} is a meromorphic functions with a simple pole at \( \infty^{(0)} \), a simple zero at \( \infty^{(1)} \), and otherwise non-vanishing and finite that satisfies \eqref{Phi-jump}. As only holomorphic functions on \( \RS \) are constants, the normalization  at \( \boldsymbol a_3 \) yields \eqref{Phitheta}. 
\end{proof}

\begin{lem}
Let \( l_0,l_1,m_0,m_1 \) be given by \eqref{lms}. Then it holds that
\begin{equation}
\label{Phizk}
\left\{
\begin{array}{ll}
\Phi(\z_0) & = (-1)^{l_0+m_0}e^{-\pi\ic(c_\rho-\mathsf K_+)}\theta(c_\rho+2\mathsf K_-)/\theta(c_\rho), \medskip \\
\Phi(\z_1) & = (-1)^{l_1+m_1}e^{-\pi\ic(c_\rho+\mathsf K_+)}\theta(c_\rho)/\theta(c_\rho+2\mathsf K_+).
\end{array}
\right.
\end{equation}
In particular, when \( |\pi(\z_k)|<\infty \), it holds that
\begin{equation}
\label{ProdPhizk}
\Phi(\z_0)\Phi(\z_1) = -(-1)^{l_0-l_1+m_0-m_1}.
\end{equation}
Moreover, we have that
\begin{equation}
\label{Phi01}
\Phi\big({\boldsymbol 0}\big) = e^{\pi\ic \mathsf K_-}\theta(1/2)/\theta(\mathsf B/2).
\end{equation}
\end{lem}
\begin{proof}
Since \( -2\mathsf K_+=2\mathsf K_- -1 \), we get from \eqref{Phitheta} that
\[
\Phi(\z_0) = e^{\pi\ic(\mathsf K_+-c_\rho-l_0-m_0\mathsf B)}\frac{\theta(c_\rho+2\mathsf K_-+m_0\mathsf B)}{\theta(c_\rho+m_0\mathsf B)}.
\]
The first relation in \eqref{Phizk} now follows from \eqref{theta-periods}. Similarly, we have that
\[
\Phi(\z_1) = e^{\pi\ic(-\mathsf K_+-c_\rho-l_1-m_1\mathsf B)}\frac{\theta(c_\rho+m_1\mathsf B)}{\theta(c_\rho+2\mathsf K_++m_1\mathsf B)},
\]
which yields the second relation in \eqref{Phizk}, again by \eqref{theta-periods}. To get \eqref{ProdPhizk}, observe that
\[
\theta(c_\rho+2\mathsf K_-) = \theta(c_\rho+2\mathsf K_+-\mathsf B)=-e^{2\pi\ic c_\rho} \theta(c_\rho+2\mathsf K_+)
\]
by \eqref{theta-periods}. Finally, \eqref{Phi01} follows from \eqref{Phitheta} and \eqref{0-integral}.
\end{proof}

\begin{lem}
Let
\begin{equation}
\label{product-Psis}
X_n:=\lim_{z\to\infty}z^{-2}\Psi_n\big(z^{(0)}\big)\Psi_{n-1}\big(z^{(1)}\big).
\end{equation}
When \( |\pi(\z_k)|<\infty \), it holds that
\begin{equation}
\label{product-Psis-V}
X_n = \frac{4}{a^2+b^2}\frac{\theta^2(c_\rho)}{\theta^2(0)} \frac{(-1)^{\imath(n)}}{\Phi^{2\imath(n)}(\z_1)}.
\end{equation}
\end{lem}
\begin{proof}
Since \( \Phi(\z)\Phi(\z^*)\equiv1 \) and \( S_\rho(\z)S_\rho(\z^*)\equiv1 \), the desired limit is equal to
\[
\frac{4}{a^2+b^2}T_{\imath(n)}\big(\infty^{(0)}\big)\lim_{z\to\infty}\Phi\big(z^{(1)}\big)T_{\imath(n-1)}\big(z^{(1)}\big),
\]
where we also used \eqref{log-cap}. Since \( -2\mathsf K_+=2\mathsf K_- -1 \), it follows from \eqref{T0T1} and \eqref{infinite-integral} that
\[
T_{\imath(n)}\big(\infty^{(0)}\big) = e^{ \pi \ic \imath(n)\mathsf K_+}\dfrac{\theta \big( c_\rho + 2\imath(n)\mathsf K_- \big)}{\theta(0)}.
\]
We further deduce from \eqref{T0T1} and \eqref{Phitheta} that
\[
\big(\Phi T_{\imath(n-1)}\big)(\z) = \exp \left \{ -\pi \ic \imath(n) \int_{\boldsymbol a_3}^{\z} \Omega \right \}  \dfrac{\theta \big( \int_{\boldsymbol a_3}^{\z} \Omega - c_\rho + (-1)^{\imath(n)}\mathsf K_+ \big)}{\theta \big( \int_{\boldsymbol a_3}^{\z} \Omega + \mathsf K_+ \big)}.
\]
Therefore, it follows from \eqref{infinite-integral} that
\[
\big(\Phi T_{\imath(n-1)}\big)\big(\infty^{(1)}\big) = e^{ \pi \ic \imath(n)\mathsf K_+}\dfrac{\theta \big( c_\rho + 2\imath(n)\mathsf K_+ \big)}{\theta(0)}.
\]
Hence, we get from \eqref{Phizk} that
\[
X_n = \frac{4}{a^2+b^2}\frac{\theta^2(c_\rho)}{\theta^2(0)} \left( (-1)^{l_0-l_1+m_0-m_1} \frac{\Phi(\z_0)}{\Phi(\z_1)}\right)^{\imath(n)}.
\]
The claim of the lemma now follows from \eqref{ProdPhizk}.
\end{proof}

\begin{lem}
It holds that
\begin{equation}
\label{theta-derivative}
\frac{\dd}{\dd\zeta}\left(e^{\pi\ic\zeta}\frac{\theta(\zeta+\mathsf K_+)}{\theta(\zeta-\mathsf K_+)} \right) = \ic\pi\theta^2(0)e^{\pi\ic\zeta}\frac{\theta(\zeta-\mathsf K_-)\theta(\zeta+\mathsf K_-)}{\theta^2(\zeta-\mathsf K_+)}.
\end{equation}
\end{lem}
\begin{proof}
See \cite[Eq. (20.7.25)]{DLMF} (observe that \( \theta(\zeta)=\theta_3(\pi\zeta|\mathsf B) \) in the notation of \cite[Chapter~20]{DLMF}).
\end{proof}

\begin{lem}
It holds that
\begin{equation}
\label{z-theta}
z= -\frac{\sqrt{a^2+b^2}}2\frac{e^{-\pi\ic\mathsf K_+}\theta^2(0)}{\theta(1/2)\theta(\mathsf B/2)} \frac{\theta\big(\int_{\boldsymbol a_3}^\z\Omega-\mathsf K_-\big)\theta\big(\int_{\boldsymbol a_3}^\z\Omega+\mathsf K_-\big)}{\theta\big(\int_{\boldsymbol a_3}^\z\Omega-\mathsf K_+\big)\theta\big(\int_{\boldsymbol a_3}^\z\Omega+\mathsf K_+\big)}.
\end{equation}
\end{lem}
\begin{proof}
It follows from \eqref{theta-periods}, \eqref{infinite-integral}, and \eqref{0-integral} that
\[
z = C \frac{\theta\big(\int_{\boldsymbol a_3}^\z\Omega-\mathsf K_-\big)\theta\big(\int_{\boldsymbol a_3}^\z\Omega+\mathsf K_-\big)}{\theta\big(\int_{\boldsymbol a_3}^\z\Omega-\mathsf K_+\big)\theta\big(\int_{\boldsymbol a_3}^\z\Omega+\mathsf K_+\big)}
\]
for some normalizing constant \( C \). It further follows from \eqref{log-cap}, \eqref{Phitheta}, and \eqref{infinite-integral} that
\[
-\frac{\sqrt{a^2+b^2}}2 = \lim_{z\to\infty} z\Phi^{-1}\big(z^{(0)}\big) = Ce^{\pi\ic \mathsf K_+}\frac{\theta(1/2)\theta(\mathsf B/2)}{\theta^2(0)},
\]
which yields the desired result.
\end{proof}

\begin{lem}
It holds that
\begin{equation}
\label{moduli}
e^{\pi\ic\mathsf B/2}\frac{\theta^2(1/2)\theta^2(\mathsf B/2)}{\theta^4(0)} = \frac{a^2+b^2}{4ab}.
\end{equation}
\end{lem}
\begin{proof}
To prove \eqref{moduli}, evaluate \eqref{z-theta} at \( \boldsymbol a_3 \) to get
\[
\frac{\theta(1/2)\theta(\mathsf B/2)}{\theta^2(0)} = \frac{\sqrt{a^2+b^2}}{2a}e^{-\pi\ic\mathsf K_+}\frac{\theta^2(\mathsf K_-)}{\theta^2(\mathsf K_+)}.
\]
Since \( \boldsymbol\Delta_3 - \boldsymbol\Delta _1 \) is homologous to \( \ualpha-\ubeta \), one can easily deduce from Figure~\ref{f:cross} that it also holds that
\[
\int_{\boldsymbol a_3}^{\boldsymbol a_2}\Omega = \left(\int_{\boldsymbol a_3}^{{\boldsymbol 0^*}} + \int_{{\boldsymbol 0^*}}^{\boldsymbol a_1}+\int_{\boldsymbol a_1}^{\boldsymbol a_2}\right)\Omega = \frac12\int_{\boldsymbol \Delta_3-\boldsymbol \Delta_1+\boldsymbol\beta}\Omega = \frac12,
\]
where the initial path of integration (except for \( \boldsymbol a_2 \)) belongs to \( \RS_{\ualpha,\ubeta} \). Thus, evaluating \eqref{z-theta} at \( \boldsymbol a_2 \) gives us
\[
\frac{\theta(1/2)\theta(\mathsf B/2)}{\theta^2(0)} = -\frac{\sqrt{a^2+b^2}}{2\ic b}e^{-\pi\ic\mathsf K_+}\frac{\theta^2(\mathsf K_+)}{\theta^2(\mathsf K_-)},
\]
where we used \eqref{theta-periods}. Multiplying two expressions for \( \theta(1/2)\theta(\mathsf B/2)/\theta^2(0) \) yields the desired result.
\end{proof}

\begin{lem}
It holds that
\begin{equation}
\label{alpha-period}
\oint_{\boldsymbol\alpha}\frac{\dd s}{w(\s)}  = \frac{2\pi\ic}{\sqrt{a^2+b^2}}e^{\pi\ic\mathsf K_+}\theta(1/2)\theta(\mathsf B/2).
\end{equation}
\end{lem}
\begin{proof}
We can deduce from \eqref{Phitheta}, \eqref{theta-derivative}, and the evenness of the theta function that
\[
\Phi^\prime(\z) = -\ic\pi\theta^2(0)\left(\oint_{\boldsymbol\alpha}\frac{\dd s}{w(\s)}\right)^{-1}\frac{\Phi(\z)}{w(\z)}\frac{\theta\big(\int_{\boldsymbol a_3}^\z\Omega+\mathsf K_-\big)\theta\big(\int_{\boldsymbol a_3}^\z\Omega-\mathsf K_-\big)}{\theta\big(\int_{\boldsymbol a_3}^\z\Omega+\mathsf K_+\big)\theta\big(\int_{\boldsymbol a_3}^\z\Omega-\mathsf K_+\big)}.
\]
Since \( \Phi^\prime(\z) = z\Phi(\z)/w(\z) \) by \eqref{Phi}, \eqref{alpha-period} follows from \eqref{z-theta}.
\end{proof}

\begin{lem}
\label{lem:4.9}
Let
\begin{equation}
\label{T-ratio}
Y_n := \big(T_{\imath(n)}^\prime T_{\imath(n-1)}/\Phi-T_{\imath(n)}(T_{\imath(n-1)}/\Phi)^\prime\big)\big({\boldsymbol 0}\big).
\end{equation}
When \( |\pi(\z_k)|=\infty \), it holds that \( Y_n = 0 \), otherwise, we have that
\begin{equation}
\label{T-ratio-V}
Y_n = (-1)^{l_0+m_0+\imath(n)}\frac{2e^{\pi\ic c_\rho}}{\sqrt{a^2+b^2}}\frac{\Phi(\z_0)}{\Phi^2\big({\boldsymbol 0}\big)}\frac{\theta^2(c_\rho)}{\theta^2(0)},
\end{equation}
where the integers \( l_0,m_0 \) were defined in \eqref{lms}.
\end{lem}
\begin{proof}
Since \( \Phi^\prime(\z) = z\Phi(\z)/w(\z) \) by \eqref{Phi}, \( \Phi^\prime\big({\boldsymbol 0}\big)=0 \). Therefore,
\[
Y_n = \big(T_{\imath(n-1)}^2 /\Phi \big)\big({\boldsymbol 0}\big) \big( T_{\imath(n)}/T_{\imath(n-1)}\big)^\prime\big({\boldsymbol 0}\big) .
\]
Assume that \( |\pi(\z_k)|<\infty \). Then it follows from \eqref{T0T1}, \eqref{theta-derivative}, and \eqref{alpha-period} that
\begin{multline*}
\left( \frac{T_{\imath(n)}}{T_{\imath(n-1)}}\right)^\prime(\z) = -(-1)^{\imath(n)}\frac{\sqrt{a^2+b^2}}{2w(\z)} \frac{e^{-\pi\ic\mathsf K_+}\theta^2(0)}{\theta(1/2)\theta(\mathsf B/2)} \left(\frac{T_{\imath(n)}}{T_{\imath(n-1)}}\right)(\z) \times \\ \times \frac{\theta\big(\int_{\boldsymbol a_3}^{\z}\Omega-c_\rho+\mathsf K_-\big)\theta\big(\int_{\boldsymbol a_3}^{\z}\Omega-c_\rho-\mathsf K_-\big)}{\theta\big(\int_{\boldsymbol a_3}^{\z}\Omega-c_\rho+\mathsf K_+\big)\theta\big(\int_{\boldsymbol a_3}^{\z}\Omega-c_\rho-\mathsf K_+\big)}.
\end{multline*}
We further deduce from \eqref{T0T1}, \eqref{0-integral}, and \eqref{Phi01} that
\[
(T_{\imath(n-1)}T_{\imath(n)})\big({\boldsymbol 0}\big) = \frac1{\Phi\big({\boldsymbol 0}\big)}\frac{\theta(c_\rho-\mathsf B/2)\theta(c_\rho+1/2)}{\theta(1/2)\theta(\mathsf B/2)}.
\]
Since \( w\big({\boldsymbol 0}\big) = \ic ab \), we therefore get from \eqref{0-integral} that
\[
Y_n = \frac{\sqrt{a^2+b^2}}{2ab} \frac{\ic(-1)^{\imath(n)}}{\Phi^2\big({\boldsymbol 0}\big)}\frac{e^{-\pi\ic\mathsf K_+}\theta^4(0)}{\theta^2(1/2)\theta^2(\mathsf B/2)}\frac{\theta(c_\rho)\theta(c_\rho+2\mathsf K_-)}{\theta^2(0)}.
\]
\eqref{T-ratio-V} now follows from \eqref{moduli} and the first formula in \eqref{Phizk}. 

Let now \( \z_0=\infty^{(1)} \), in which case \( [c_\rho]=[0] \). Since \( \Phi\big(\infty^{(1)}\big)=0 \), we get that \( Y_n=0 \). Finally, let  \( \z_1=\infty^{(1)} \). Then we have that \( -c_\rho=-(-1)^k2\mathsf K_+ + l_k + m_k\mathsf B\) and therefore
\begin{eqnarray*}
\frac{T_1(\z)}{T_0(\z)} &=& \exp\left\{\pi\ic\int_{\boldsymbol a_3}^{\z}\Omega\right\}\frac{\theta\big(\int_{\boldsymbol a_3}^{\z}\Omega+m_1\mathsf B+3\mathsf K_+\big)}{\theta\big(\int_{\boldsymbol a_3}^{\z}\Omega+(m_1+1)\mathsf B-3\mathsf K_+\big)} \\
& = & \exp\left\{\pi\ic\int_{\boldsymbol a_3}^{\z}\Omega\right\}\frac{\theta\big(\int_{\boldsymbol a_3}^{\z}\Omega+(m_1+1)\mathsf B-\mathsf K_+\big)}{\theta\big(\int_{\boldsymbol a_3}^{\z}\Omega+m_1\mathsf B+\mathsf K_+\big)} \\
& = & e^{2\pi\ic(2m_1+1)\mathsf K_-}\Phi(\z)
\end{eqnarray*}
by \eqref{theta-periods} and \eqref{Phitheta}. As \( \Phi^\prime\big({\boldsymbol 0}\big)=0 \), it also holds that \( Y_n =0 \).
\end{proof}

\begin{lem}
\label{lem:4.10}
Let
\begin{equation}
\label{T-ratio1}
Z_n := \big(T_{\imath(n)}^\prime T_{\imath(n-1)}/\Phi-T_{\imath(n)}(T_{\imath(n-1)}/\Phi)^\prime\big)\big({\boldsymbol 0^*}\big).
\end{equation}
When \( |\pi(\z_k)|=\infty \), it holds that \( Z_n = 0 \), otherwise, we have that
\begin{equation}
\label{T-ratio1-V}
Z_n = (-1)^{l_0+m_0+\imath(n)}\frac{2e^{-\pi\ic c_\rho}}{\sqrt{a^2+b^2}}\frac{\Phi(\z_0)}{\Phi^2\big({\boldsymbol 0^*}\big)}\frac{\theta^2(c_\rho)}{\theta^2(0)}.
\end{equation}
\end{lem}
\begin{proof}
The proof is the same as in the previous lemma.
\end{proof}

\begin{lem}
Let \( \sigma_0,\sigma_1 \) be as in \eqref{Arhon}. When \( |\pi(\z_k)|<\infty \), it holds that
\begin{equation}
\label{XnYn}
Y_nX_n^{-1} = \sigma_{\imath(n)}e^{\pi\ic c_\rho}\frac{\sqrt{a^2+b^2}}2\frac{\Phi\big(\z_{\imath(n)}\big)}{\Phi^2\big({\boldsymbol 0}\big)}
\end{equation}
and
\begin{equation}
\label{XnZn}
Z_nX_n^{-1} = \sigma_{\imath(n)}e^{-\pi\ic c_\rho}\frac{\sqrt{a^2+b^2}}2\frac{\Phi\big(\z_{\imath(n)}\big)}{\Phi^2\big({\boldsymbol 0^*}\big)},
\end{equation}
where \( X_n \), \( Y_n \), and \( Z_n \) are given by \eqref{product-Psis}, \eqref{T-ratio}, and \eqref{T-ratio1}, respectively.
\end{lem}
\begin{proof}
The claims follow immediately from \eqref{product-Psis-V}, \eqref{T-ratio-V}, \eqref{T-ratio1-V}, and \eqref{ProdPhizk}.
\end{proof}

\section{Proof of Proposition~\ref{prop:szego}}
\label{sec:5}

It follows from \eqref{cauchy-kernel} that \( \Omega_{\z, \z^*} = - \Omega_{\z^*, \z} \) for all $\z \in \RS$ such that $\pi(\z) \in \C$ and therefore \( S_{\rho}(\z) S_{\rho}(\z^*) \equiv 1 \) for such \( \z \). Clearly,  this relation extends to the points on top of infinity by continuity. It is also immediate from 
\eqref{szego} and \eqref{cauchy-kernel} that
\begin{multline}
\label{szego-flat}
S_{\rho}\big(z^{(0)}\big) = \exp \left \{ - \sum_{i = 1}^{4} \dfrac{w(z)}{2 \pi\ic} \int_{\Delta_i} \dfrac{\log (\rho_i w_+)(s)}{s - z} \dfrac{\dd s}{w_{|\Delta_i+}(s)} \right \} \times \\ \times \exp \big \{ 2\pi \ic (wH)(z) c_{\rho}\big \},
\end{multline}
where, for emphasis, we write $w_{|\Delta_i+}(s)$ for \( w_+(s) \)  on $s\in\Delta_i^\circ$ and
\begin{equation}
\label{H}
H(z) := \dfrac{1}{2 \pi \ic} \int_{\pi(\ualpha)} \dfrac{\dd t}{(t - z) w(t)}.
\end{equation}
Relations \eqref{S-jump} now easily follow from \eqref{szego-flat},  \eqref{H}, and Plemelj-Sokhotski formulae \cite[equations (4.9)]{Gakhov}. As for the behavior near $a_i$, note that by \cite[equation (8.8)]{Gakhov}, the function $(wH)(z)$ is bounded as $z \to a_i$. Furthermore, \cite[equations (8.8) and (8.35)]{Gakhov} yield that
\[
-\dfrac{w(z)}{2\pi \ic} \int_{\Delta_i} \dfrac{\log (\rho_i w_+)(s)}{s - z} \dfrac{\dd s}{w_{|\Delta_i+}(s)} = -\dfrac{1}{2} \log (z - a_i)^{\alpha_i + 1/2} + \mathcal O(1).
\]
Since the above integral is the only one with the singular contribution  around \( a_i \), the validity of the top line in \eqref{Srho-ai} follows. As for the behavior near the origin, note that \( \lim_{\mathcal Q_j\in z\to0} w(z) = (-1)^{j-1}\ic ab \), where, as before, \( \mathcal Q_j \) stands for the \( j \)-th quadrant. Recall that each segment \( \Delta_i \) is oriented towards the origin, see Figure~\ref{f:cross}. Hence, it follows from  \cite[equation (8.2)]{Gakhov} that
\begin{multline*}
-\dfrac{w(z)}{2\pi \ic} \int_{\Delta_i}  \dfrac{\log (\rho_i w_+)(s)}{s - z} \dfrac{\dd s}{w_{\Delta_i + }(s)} = -\dfrac{w(z)}{2 \pi \ic}\dfrac{\log (\rho_i w_+)(0)}{w_{|\Delta_i +}(0)} \log (z) + F_i(z) \\ = \frac{(-1)^{j+i}}{2\pi\ic}\log (\rho_i w_+)(0)\log (z) + F_i(z), \quad z\in \mathcal Q_j,
\end{multline*}
where \( F_i(z) \) is a bounded function around the origin tending to a definite limit as \( z\to 0 \). Thus, summing over $i$ yields
\[
-\dfrac{w(z)}{2\pi \ic} \int_\Delta  \dfrac{\log (\rho_i w_+)(s)}{s - z} \dfrac{\dd s}{w_+(s)} = (-1)^j\nu\log(z) + \sum_{i=1}^4F_i(z), \quad z\in \mathcal Q_j,
\]
where \( \nu \) was defined in \eqref{nu} and we used \eqref{logw}. Since $(wH)(z)$ is holomorphic around the origin, the second line in \eqref{Srho-ai} follows.

\section{Proofs of Theorems~\ref{thm:asymptotics} and~\ref{thm:pade}}
\label{sec:6}

\subsection{Initial RH problem}
Just as was first done by Fokas, Its, and Kitaev \cite{FIK91,FIK92}, we connect the orthogonal polynomials $Q_n(z)$ to a \( 2\times2 \) matrix Riemann-Hilbert problem. To this end, suppose that the index $n$ is such that
\begin{equation}
\label{normal-index}
\deg Q_n = n \qandq R_{n-1}(z) \sim z^{-n} \qasq z \to \infty,
\end{equation}
where $R_n(z)$ is given by \eqref{linear-system}. Let
\begin{equation}
\label{Y-matrix}
\boldsymbol Y(z) := \left(  \begin{matrix}
Q_n(z) & R_n(z)\\
k_{n-1} Q_{n-1}(z) & k_{n-1}R_{n-1}(z)
\end{matrix} \right),
\end{equation}
where $k_{n-1}$ is a constant such that $k_{n-1}R_{n-1}(z) = z^{-n} (1 + o(1))$ near infinity. Then $\boldsymbol Y(z)$ solves the following Riemann-Hilbert problem (\rhy):
\begin{itemize}
\label{rhy}
	\item[(a)] $\boldsymbol Y(z)$ is analytic in $\C \setminus \Delta$ and $\lim_{z \to \infty} \boldsymbol Y(z)z^{-n \sigma_3} = \boldsymbol I$ \footnote{Hereafter, we set $\sigma_3 := \left(\begin{matrix} 1 & 0 \\ 0 & -1 \end{matrix}\right)$ and \( \boldsymbol I \) to be the identity matrix.}.
	\item[(b)] $\boldsymbol Y(z)$ has continuous traces on $\Delta^\circ$ that satisfy 
	\[
	\boldsymbol Y_{+}(s) = \boldsymbol Y_- (s)\left( \begin{matrix}
	1 & \rho(s) \\ 0 & 1
	\end{matrix} \right), \quad s\in\Delta^\circ.
	\]
	\item[(c)] \( \boldsymbol Y(z) \) is bounded around the origin and
	\[
	\boldsymbol Y(z) = \left \{  \begin{array}{rl} 
	\mathcal O\left(\begin{matrix}
	1 & 1\\ 1 & 1
	\end{matrix}\right)  & \text{ if } \alpha_i > 0, \medskip \\
	\mathcal O\left(\begin{matrix}
	 1 & \log |z - a_i|\\ 1 & \log |z - a_i|
	\end{matrix}\right)  & \text{ if } \alpha_i = 0,\medskip \\
	\mathcal O\left(\begin{matrix}
	1 & |z - a_i|^{\alpha_i}\\ 1 &  |z - a_i|^{\alpha_i}
	\end{matrix}\right)  & \text{ if }-1 < \alpha_i < 0,
	\end{array}\right.
	\]
	as \( z\to a_i \) for each \( i \in\{1,2,3,4\} \).
\end{itemize}
Indeed, property \hyperref[rhy]{\rhy}(a) is an immediate consequence of \eqref{normal-index}. The jump relations in \hyperref[rhy]{\rhy}(b) follow from \eqref{hatrho}, \eqref{linear-system}, and an application of the Plemelj-Sokhotski formulae. Behavior of the Cauchy integrals around the contours of integration, see  \cite[Section 8]{Gakhov} and \eqref{linear-system} yield \hyperref[rhy]{\rhy}(c) (to deduce boundedness around the origin one needs to utilize the third condition in the definition of the class \( \mathcal W_1 \)). 

On the other hand, it also can be shown that if a solution of \hyperref[rhy]{\rhy} exists, then it must be of the form \eqref{Y-matrix} with the diagonal entries satisfying \eqref{normal-index} (see, for example, \cite[Lemma 1]{ApY15}).

In what follows we prove solvability of \hyperref[rhy]{\rhy} for all \( n\in\N_{\rho,\varepsilon} \) large enough via the matrix steepest descent method developed by Deift and Zhou~\cite{DZ93}.

\subsection{Opening of the Lenses}
\label{ss:ol}

Let \( \delta_0>0 \) be small enough so that all the functions \( \rho_i^*(z) \) are holomorphic in some neighborhood of \( \{|z|\leq \delta_0\} \). Define \( \tilde\Delta_i \) and \( \tilde\Delta_i^\circ \) to be the closed and open segments connecting the origin and \( \delta_0e^{(2i-1)\pi\ic/4} \), \( i\in\{1,2,3,4\} \), that are oriented towards the origin.
\begin{figure}[!ht]
	
\begin{tikzpicture}
\draw (3, 0) -- (-3,0);
\draw (0, 3) -- (0, -3);
\node[rotate = 90] at (0, 1){$<$};
\node[rotate = 90] at (0, -1){$>$};
\node at (1, 0){$<$};
\node at (-1, 0){$>$};
%
\draw (0.75, 0.75) -- (-0.75, -0.75);
\draw (0.75, -0.75) -- (-0.75, 0.75);
\draw (0.75, 0.75) to[out = 45, in = 135] (3, 0);
\draw (0.75, -0.75) to[out = -45, in = -135] (3, 0);
\draw (-0.75, 0.75) to[out = 135, in = 45] (-3, 0);
\draw (-0.75, -0.75) to[out = -135, in = -45] (-3, 0);
\draw (0.75, 0.75) to[out = 45, in = -45] (0, 3);
\draw (-0.75, 0.75) to[out = 135, in = -135] (0, 3);
\draw (0.75, -0.75) to[out = -45, in = 45] (0, -3);
\draw (-0.75, -0.75) to[out = -135, in = 135] (0, -3);
%
\node at (3, 0){\textbullet}; \node[right] at (3, 0){$a_1$};
\node at (-3, 0){\textbullet}; \node[left] at (-3, 0){$a_3$};
\node at (0, 3){\textbullet}; \node[above] at (0, 3){$a_2$};
\node at (0, -3){\textbullet}; \node[below] at (0, -3){$a_4$};
\node at (0.75, 0.75){\textbullet}; 
\node at (-0.75, 0.75){\textbullet}; 
\node at (0.75, -0.75){\textbullet}; 
\node at (-0.75, -0.75){\textbullet}; 
\node at (2.5, 1){$\Gamma_{1-}$};
\node at (2.5, -1){$\Gamma_{1+}$};
\node at (1, 2.5){$\Gamma_{2+}$};
\node at (-1, 2.5){$\Gamma_{2-}$};
\node at (-2.5, 1){$\Gamma_{3+}$};
\node at (-2.5, -1){$\Gamma_{3-}$};
\node at (1, -2.5){$\Gamma_{4-}$};
\node at (-1, -2.5){$\Gamma_{4+}$};
\node[right] at (0.5, 0.4){$\tilde{\Delta}_1$};
\node[left] at (-0.5, 0.4){$\tilde{\Delta}_2$};
\node[left] at (-0.5, -0.4){$\tilde{\Delta}_3$};
\node[right] at (0.5, -0.4){$\tilde{\Delta}_4$};
\node at (0.4, 1.6){\(\Omega_{2+}\)};
\node at (-0.4, 1.6){\(\Omega_{2-}\)};
\node at (-0.4, -1.6){\(\Omega_{4+}\)};
\node at (0.4, -1.6){\(\Omega_{4-}\)};
\node at (2, -0.4){\(\Omega_{1+}\)};
\node at (2, 0.4){\(\Omega_{1-}\)};
\node at (-2, 0.4){\(\Omega_{3+}\)};
\node at (-2, -0.4){\(\Omega_{3-}\)};
\end{tikzpicture}
\caption{\small The arcs $\Delta_i, \ \tilde{\Delta}_i$ and \( \Gamma_{i\pm} \), and domains \( \Omega_{i\pm} \).}
\label{lens}
\end{figure}
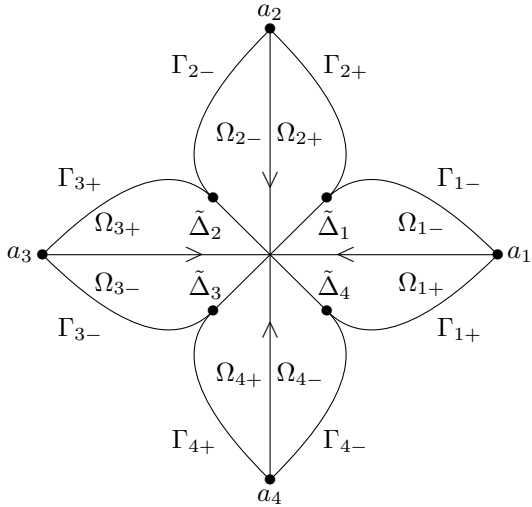
Further, let \( \Gamma_{i-},\Gamma_{i+} \) be open smooth arcs that lie within the domain of holomorphy of \( \rho_i^*(z) \) and connect \( a_i \) to \( \delta_0e^{(2i-1)\pi\ic/4},\delta_0e^{(2i-3)\pi\ic/4} \), respectively. We orient \( \Gamma_{i\pm} \) away from \( a_i \) and assume that no open arcs \( \Delta_i^\circ,\tilde\Delta_i^\circ,\Gamma_{i\pm} \) intersect, see Figure~\ref{lens}. We denote by \( \Omega_{i\pm} \) the domain partially bounded by \( \Delta_i \) and \( \Gamma_{i\pm} \). Let 
\begin{equation}
\label{eq:x}
\boldsymbol X(z) := \boldsymbol Y(z) \left\{
\begin{array}{rl}
\displaystyle  \left(\begin{matrix} 1 & 0 \smallskip \\ \mp 1/\rho_i(z) & 1 \end{matrix}\right), & z\in\Omega_{i\pm} \smallskip, \\
\boldsymbol I, & z\not\in\overline\Omega_{i+}\cup\overline\Omega_{i-}.
\end{array}
\right.
\end{equation}
Then $\boldsymbol X(z)$ satisfies the following Riemann-Hilbert problem (\rhx):
\begin{itemize}
	\label{rhx}
	\item[(a)] \( \boldsymbol X(z) \) is analytic in \( \C\setminus\cup_i(\Delta_i \cup \tilde{\Delta}_i\cup  \Gamma_{i\pm}) \) and \( \displaystyle \lim_{z\to\infty} \boldsymbol X(z)z^{-n\sigma_3} = \boldsymbol I \);
	\item[(b)] \( \boldsymbol X(z) \) has continuous traces on each \( \Delta_i^\circ, \ \tilde{\Delta}_i^\circ\), and \( \Gamma_{i\pm} \) that satisfy 
	\[
	\boldsymbol X_+(s) = \boldsymbol X_-(s)
	\left\{
	\begin{array}{rl}
	\displaystyle \left(\begin{matrix} 1 & 0 \smallskip \\ 1/ \rho_i(s) & 1 \end{matrix}\right), & s\in\Gamma_{i+}\cup\Gamma_{i-}, \medskip \\
	\displaystyle \left(\begin{matrix} 0 & \rho_i(s) \\ -1/\rho_i(s) & 0 \end{matrix} \right), & s\in\Delta_i^\circ, \medskip \\
	\left(\begin{matrix}
	1& 0 \smallskip\\
	\dfrac{1}{\rho_i(s)} + \dfrac{1}{\rho_{i + 1}(s)} & 1
	\end{matrix}\right), & s\in\tilde \Delta_i^\circ,
	\end{array} \right. \\
	\]
	where $i \in\{ 1, 2, 3, 4 \}$ and $\rho_5 := \rho_1$.
	\item[(c)] \( \boldsymbol X(z) \) is bounded around the origin and behaves like
	\[
	\boldsymbol X(z) = \left \{  \begin{array}{rl} 
	\mathcal O\left(\begin{matrix}
	1 & 1\\ 1 & 1
	\end{matrix}\right)  & \text{ if } \alpha_i > 0, \medskip \\
	\mathcal O\left(\begin{matrix}
	1 & \log |z - a_i|\\ 1 & \log |z - a_i|
	\end{matrix}\right)  & \text{ if } \alpha_i = 0,\medskip \\
	\mathcal O\left(\begin{matrix}
	1 & |z - a_i|^{\alpha_i}\\ 1 &  |z - a_i|^{\alpha_i}
	\end{matrix}\right)  & \text{ if }-1 < \alpha_i < 0,
	\end{array}\right.
	\]  as $z \to a_i$ from outside the lens while from inside the lens, 
	\[
	\boldsymbol X(z) = \left \{ \begin{array}{rl} 
	\mathcal O\left(\begin{matrix} |z - a_i|^{-\alpha_i} & 1 \smallskip \\ |z - a_i|^{-\alpha_i} & 1\end{matrix}\right) & \text{if } \alpha_i >0, \medskip \\
	\mathcal O\left(\begin{matrix}
	1 & \log |z - a_i|\\ 1 & \log |z - a_i|
	\end{matrix}\right)  & \text{ if } \alpha_i = 0,\medskip \\
	\mathcal O\left(\begin{matrix} 1 & |z-a_i|^{\alpha_i} \smallskip \\ 1 & |z-a_i|^{\alpha_i}\end{matrix}\right) & \text{if } -1 < \alpha_i < 0.
	
	\end{array} \right.
	\]
	
\end{itemize}

The following observation can be easily checked: \hyperref[rhx]{\rhx} is solvable if and only if \hyperref[rhy]{\rhy} is solvable. When solutions of \hyperref[rhx]{\rhx} and \hyperref[rhy]{\rhy} exist, they are unique and connected by \eqref{eq:x}.

\subsection{Global Parametrix}
\label{sec:6.3}

Let \( \Psi_n(\z) \) be given by \eqref{Psin}. For each \( n\in\N_{\rho,\varepsilon} \), define
\begin{equation}
\label{N}
\boldsymbol N(z) := \left(\begin{matrix}
\gamma_n & 0\\
0 & \gamma^*_{n-1}
\end{matrix}\right)\left(\begin{matrix} \Psi_n\big(z^{(0)}\big) & \Psi_n\big(z^{(1)}\big)/ w(z)\\
\Psi_{n-1}\big(z^{(0)}\big)  & \Psi_{n-1}\big(z^{(1)}\big) / w(z)
\end{matrix}\right),
\end{equation}
where the constants \( \gamma_n \) and \( \gamma_{n-1}^* \) are defined by the relations
\begin{equation}
\label{gamma-constants}
\lim_{z \to \infty} \gamma_nz^{-n} \Psi_n\big(z^{(0)}\big)=1 \qandq  \lim_{z \to \infty} \gamma_{n-1}^*z^{n} \Psi_{n-1}\big(z^{(1)}\big)/w(z) =1.
\end{equation}
Such constants do exist by the very definition of the sequence \( \N_{\rho,\varepsilon} \) in \eqref{Nrhoepsilon} and the fact that the above normalization of \( \Psi_{n-1}(\z) \) is possible if and only if the above normalization of \( \Psi_n(\z) \) is possible, see the proof of Proposition~\ref{prop:N}. The product $\gamma_n\gamma^*_{n-1}$ assumes only two necessarily finite and non-zero values depending on the parity of \( n \) (when \( |\pi(\z_k)|<\infty \), it is equal to \( X_n^{-1} \), see \eqref{product-Psis}). The matrix \( \boldsymbol N(z) \) solves the following Riemann-Hilbert problem (\rhn):
\begin{itemize}
\label{rhn}
\item[(a)] \( \boldsymbol N(z) \) is analytic in \( \C\setminus\Delta \) and \( \displaystyle \lim_{z\to\infty} \boldsymbol N(z)z^{-n\sigma_3}= \boldsymbol I\);
\item[(b)] \( \boldsymbol N(z) \) has continuous traces on \( \Delta^\circ \) that satisfy 
\[
\boldsymbol N_+(s) = \boldsymbol N_- (s)
\displaystyle \left(\begin{matrix} 0 & \rho(s) \\ -1/ \rho(s) & 0\end{matrix}\right), \quad s\in\Delta^\circ;
\]
\item[(c)] \( \boldsymbol N(z) \) satisfies
\[
\boldsymbol N(z) =  \mathcal O\left(\begin{matrix} |z-a_i|^{-(2\alpha_i + 1)/4} & |z-a_i|^{(2\alpha_i - 1)/4} \smallskip \\ |z-a_i|^{-(2\alpha_i + 1)/4} & |z-a_i|^{(2\alpha_i - 1)/4}\end{matrix}\right) \qasq z\to a_i,
\]
$i\in \{1, 2, 3, 4 \}$, and 
\[
\boldsymbol N(z) =  \mathcal O\left(\begin{matrix} |z|^{(-1)^j \mathrm{Re}(\nu)} & |z|^{(-1)^{j + 1} \mathrm{Re}(\nu)} \smallskip \\ |z|^{(-1)^j \mathrm{Re}(\nu)} & |z|^{(-1)^{j + 1} \mathrm{Re}(\nu) }\end{matrix}\right) \qasq z \to 0,
\]
where $j\in \{1, 2, 3, 4 \}$ is the number of the quadrant from which $z \to 0$ an $\nu$ is given by \eqref{nu}.
\end{itemize}
Indeed, \hyperref[rhn]{\rhn}(a) holds by construction, while \hyperref[rhn]{\rhn}(b,c) follow from \eqref{Psin-jump} and \eqref{Srho-ai}, respectively (notice that the actual exponents in \hyperref[rhn]{\rhn}(c) will change when the considered point happens to coincide with \( z_{\imath(n)} \) or \( z_{\imath(n-1)} \)). Notice also that \( \det(\boldsymbol N(z)) \equiv 1 \) since this is an entire function (it clearly has no jumps and it can have at most square root singularities at the points \( a_i \)) that converges to 1 at infinity.

For later calculations it will be convenient to set
\begin{equation}
\label{M}
\boldsymbol M^\star(z) :=  \left(\begin{matrix} (S_{\rho}T_{\imath(n)})(z^{(0)}) & (S_{\rho}T_{\imath(n)})(z^{(1)})/ w(z)\\
(S_{\rho}T_{\imath(n-1)}/\Phi)(z^{(0)}) & (S_{\rho}T_{\imath(n-1)}/\Phi)(z^{(1)}) /  w(z)
\end{matrix}\right),
\end{equation}
and \( \boldsymbol M(z) := \left( \boldsymbol I + \boldsymbol L_\nu/z\right)\boldsymbol M^\star(z) \), where \( \boldsymbol L_\nu \) is a certain constant matrix with zero trace and determinant defined further below in \eqref{defL}. Observe that \( \boldsymbol N(z)=\boldsymbol C \boldsymbol M^\star(z) \boldsymbol D(z) \), where
\begin{equation}
\label{CD}
\boldsymbol C := \left(\begin{matrix}
\gamma_n & 0\\
0 & \gamma^*_{n-1}
\end{matrix}\right) \qandq \boldsymbol D(z):= \Phi^{n \sigma_3}\big(z^{(0)}\big).
\end{equation}
When \( \re(\nu)\in(-1/2,1/2) \), it is possible to take \( \boldsymbol L_\nu \) to be the zero matrix, but this would worsen the error rates in \eqref{Qnasymp} and \eqref{Rnasymp}. When \( \re(\nu)=1/2 \), our analysis necessitates introduction of \( \boldsymbol L_\nu \). Notice that neither the normalization of \( \boldsymbol M(z) \) at infinity nor its determinate depend on \( \boldsymbol L_\nu \). In fact, it holds that \( \det(\boldsymbol M(z)) = \det(\boldsymbol M^\star(z)) = (\gamma_n\gamma_{n-1}^*)^{-1} \).

\subsection{Local Parametrix around \( a_i \)}

Let \( U_i \) be a disk around \( a_i \) of small enough radius so that \( \rho_i^*(z) \) is holomorphic around \( \overline U_i \), \( i\in\{1,2,3,4\} \). In this section we construct solution of \hyperref[rhx]{\rhx} locally in each \( U_i \). More precisely, we seeking a solution of the following local Riemann-Hilbert problem (\rhpi):
\begin{itemize}
\label{rhp}
\item[(a,b,c)] $\boldsymbol P_{a_i}(z)$ satisfies \hyperref[rhx]{\rhx}(a,b,c) within \( U_i\);
\item[(d)] \( \boldsymbol P_{a_i}(s) = \boldsymbol M(s)\big(\boldsymbol I+\boldsymbol{\mathcal O}(1/n) \big)\boldsymbol D(s) \) uniformly for \( s\in\partial U_i \).
\end{itemize}

We shall only construct a solution of \hyperref[rhp]{\rhp} as other constructions are almost identical.

\subsubsection{Model Problem}

Below, we always assume that the real line as well as its subintervals is oriented from left to right. Further, we set
\[
I_\pm:=\big\{z:~\arg(\zeta)=\pm2\pi/3\big\},
\]
where the rays $I_\pm$ are oriented towards the origin. Given $\alpha>-1$, let $\boldsymbol\Psi_\alpha(\zeta)$ be a matrix-valued function such that 
\begin{itemize}
\label{rhpsiA}
\item[(a)] $\boldsymbol\Psi_\alpha(\zeta)$ is analytic in $\C\setminus\big(I_+\cup I_-\cup(-\infty,0]\big)$;
\item[(b)] $\boldsymbol\Psi_\alpha(\zeta)$ has continuous traces on $I_+\cup I_-\cup(-\infty,0)$ that satisfy
\[
\boldsymbol\Psi_{\alpha+} = \boldsymbol\Psi_{\alpha-}
\left\{
\begin{array}{rll}
\left(\begin{matrix} 0 & 1 \\ -1 & 0 \end{matrix}\right) & \text{on} & (-\infty,0), \medskip \\
\left(\begin{matrix} 1 & 0 \\ e^{\pm\pi\mathrm{i}\alpha} & 1 \end{matrix}\right) & \text{on} & I_\pm;
\end{array}
\right.
\]
\item[(c)] as $\zeta\to0$  it holds that
\[
\boldsymbol\Psi_\alpha(\zeta) = \mathcal{O}\left( \begin{matrix} |\zeta|^{\alpha/2} & |\zeta|^{\alpha/2} \\ |\zeta|^{\alpha/2} & |\zeta|^{\alpha/2} \end{matrix} \right) \quad \text{and} \quad \boldsymbol\Psi_\alpha(\zeta) = \mathcal{O}\left( \begin{matrix} \log|\zeta| & \log|\zeta| \\ \log|\zeta| & \log|\zeta| \end{matrix} \right) 
\]
when $\alpha<0$ and $\alpha=0$, respectively, and
\[
\boldsymbol\Psi_\alpha(\zeta) = \mathcal{O}\left( \begin{matrix} |\zeta|^{\alpha/2} & |\zeta|^{-\alpha/2} \\ |\zeta|^{\alpha/2} & |\zeta|^{-\alpha/2} \end{matrix} \right) \quad \text{and} \quad \boldsymbol\Psi_\alpha(\zeta) = \mathcal{O}\left( \begin{matrix} |\zeta|^{-\alpha/2} & |\zeta|^{-\alpha/2} \\ |\zeta|^{-\alpha/2} & |\zeta|^{-\alpha/2} \end{matrix} \right)
\]
when $\alpha>0$, for $|\arg(\zeta)|<2\pi/3$ and $2\pi/3<|\arg(\zeta)|<\pi$, respectively;
\item[(d)] it holds uniformly in $\C\setminus\big(I_+\cup I_-\cup(-\infty,0]\big)$ that
\[
\boldsymbol\Psi_\alpha(\zeta) = \boldsymbol S(\zeta)\left(\boldsymbol I+\boldsymbol{\mathcal O}\left(\zeta^{-1/2}\right)\right)\exp\left\{2\zeta^{1/2}\sigma_3\right\},
\]
where \( \displaystyle \boldsymbol S(\zeta) := \frac{\zeta^{-\sigma_3/4}}{\sqrt2}\left(\begin{matrix} 1 & \ic \\ \ic & 1 \end{matrix}\right) \) and we take the principal branch of \( \zeta^{1/4} \).
\end{itemize}
Explicit construction of this matrix can be found in \cite{KMcLVAV04} (it uses modified Bessel and Hankel functions). Observe that
\begin{equation}
\label{S}
\boldsymbol S_+(\zeta) = \boldsymbol S_-(\zeta)\left(\begin{matrix} 0 & 1 \\ \ -1 & 0 \end{matrix}\right),
\end{equation}
since the principal branch of \( \zeta^{1/4} \) satisfies $\zeta_+^{1/4}=\ic\zeta^{1/4}_-$. Also notice that the matrix \( \sigma_3\boldsymbol\Psi_\alpha(\zeta)\sigma_3 \) satisfies \hyperref[rhpsiA]{\rhpsiA} only with the reversed orientation of \( (-\infty,0] \) and  \( I_\pm \).

\subsubsection{Conformal Map}

Since \( w(z) \) has a square root singularity at \( a_1 \) and satisfies $w_+(s)=-w_-(s)$, $s\in\Delta$, the function
\begin{equation}
\label{zeta-a-1}
\zeta_{a_1}(z) := \left(\dfrac{1}{2}\int_{a_1}^z\frac{s\dd s}{w(s)}\right)^2, \quad z\in U_1,
\end{equation}
is holomorphic in $U_1$ with a simple zero at $a_1$. Thus, the radius of $U_1$ can be made small enough so that $\zeta_{a_1}(z)$ is conformal on $\overline U_1$. Observe that \( s\dd s/w_\pm(s) \) is purely imaginary on \( \Delta_1^\circ \) and  therefore \( \zeta_{a_1}(z) \) maps \( \Delta_1 \cap U_1 \) into the negative reals. It is also rather obvious that \( \zeta_{a_1}(z) \) maps the interval \( (a_1,\infty)\cap U_1 \) into the positive reals. As we have had some freedom in choosing the arcs \( \Gamma_{1 \pm} \), we shall choose them within \( U_1 \) so that \( \Gamma_{1-} \) is mapped into \( I_+ \) and \( \Gamma_{1+} \) is mapped into \( I_- \). Notice that the orientation of the images of \( \Delta_1,\Gamma_{1+},\Gamma_{1-} \) under \( \zeta_{a_1}(z) \) are opposite from the ones of \( (-\infty,0],I_-,I_+ \).

In what follows, we understand that \( \zeta_{a_1}^{1/2}(z) \) stands for the branch given by the expression in the parenthesis in \eqref{zeta-a-1}. 

\subsubsection{Matrix $\boldsymbol P_{a_1}$}

According to the definition of the class \( \mathcal W_1 \), it holds that
\[
\rho(z)=\rho_1^*(z)(a_1 - z)^{\alpha_{1}}, \quad z\in U_1,
\]
where $\rho_1^*(z)$ is non-vanishing and holomorphic in $U_{1}$ and \( (a_1-z)^{\alpha_1} \) is the branch holomorphic in \( U_{1}\setminus [a_1,\infty)\) and positive on \( \Delta_1 \). Define
\[
r_{a_1}(z) := \sqrt{\rho_1^*(z)}(z - a_{1})^{\alpha_1/2}, \quad z\in U_1\setminus\Delta_1,
\]
where \( (z-a_1)^{\alpha_1/2} \) is the principal branch. It clearly holds that
\[
 (z - a_1)^{\alpha_1} = e^{\pm\pi\ic\alpha_1}(a_1-z)^{\alpha_1}, \quad z\in U_{1}^\pm,
\]
where \( U_{1}^\pm:= U_1\cap \{\pm\im(z)>0\} \). Then 
\[
\left\{
\begin{array}{ll}
r_{a_1+}(s)r_{a_1-}(s) = \rho(s), & s\in \Delta_1 \cap U_1, \medskip \\
r_{a_1}^2(z) = \rho(z)e^{\pm\pi\ic\alpha_1}, & z\in U^{\pm}_{1 }.
\end{array}
\right.
\]
The above relations and \hyperref[rhpsiA]{\rhpsiA}(a,b,c) imply that
\begin{equation}
\label{Pa}
\boldsymbol P_{a_1}(z) := \boldsymbol E_{a_1}(z) \sigma_3\boldsymbol\Psi_{\alpha_1}\left(n^2\zeta_{a_1}(z)\right)\sigma_3 r_{a_1}^{-\sigma_3}(z)
\end{equation}
satisfies  \hyperref[rhp]{\rhp}(a,b,c) for any holomorphic matrix $\boldsymbol E_{a_1}(z)$.

\subsubsection{Matrix $\boldsymbol E_{a_1}$}

Now we choose \( \boldsymbol E_{a_1}(z) \) so that \hyperref[rhp]{\rhp}(d) is fulfilled. To this end, denote by \( V_1,V_2,V_3 \) the sectors within \( U_1 \) delimited by \( \pi(\ualpha)\cup\pi(\ubeta) \), \( \pi(\ubeta)\cup\Delta_1 \), and \( \Delta_1\cup\pi(\ualpha) \), respectively, see Figure~\ref{f:cross}. Let $\gamma \subset \C\setminus\Delta$ be a path from $a_3$ to $a_1$ that does not intersect $\pi(\ualpha),\pi(\ubeta)$. Further, let \( \boldsymbol\gamma:=\pi^{-1}(\gamma) \) be a cycle oriented so that \( \boldsymbol\gamma^{(0)}:=\boldsymbol \gamma\cap\RS^{(0)} \) proceeds from \( \boldsymbol a_3 \) to \( \boldsymbol a_1 \). Define
\[
K_{a_1}(z) :=  \left\{
\begin{array}{llr}
&\exp\big\{\int_{\boldsymbol\gamma^{(0)}}G\big\} = \exp\{\pi \ic \left( \tau - \omega \right) \} =  1, & z \in V_1, \medskip \\
&\exp\big\{\int_{\boldsymbol\gamma^{(0)}-\ualpha}G\big\} = \exp \{ -\pi \ic (\tau +\omega) \}= -1, & z \in V_2, \medskip \\
&\exp\big\{\int_{\boldsymbol\gamma^{(0)}-\ubeta} G\big\} = \exp \{ \pi \ic (\tau +\omega) \}=-1, & z \in V_3,
\end{array} \right.
\]
where we used the symmetry \( G(\z^*)=-G(\z) \), the fact that \( \boldsymbol \gamma \) is homologous to \( \ualpha+\ubeta \), see Figure~\ref{f:basis}, and \eqref{constants}--\eqref{omtau12}. Recalling the definition of \( \Phi(\z) \) in \eqref{Phi} (the path of integration must lie in \( \RS_{\ualpha,\ubeta} \)), one can see that
\[
\Phi\big(z^{(0)}\big) = K_{a_1}(z)\exp\big\{2\zeta_{a_1}^{1/2}(z)\big\}, \quad z\in V_1\cup V_2\cup V_3.
\]
Clearly, \( |K_{a_1}(z)|=1 \). It now follows from \hyperref[rhpsiA]{\rhpsiA}(d) that
\[
\boldsymbol P_{a_1}(s) = \boldsymbol E_{a_1}(s)\sigma_3\boldsymbol S\big(n^2\zeta_{a_1}(s)\big)\sigma_3 r_{a_1}^{-\sigma_3}(s)K_{a_1}^{-n\sigma_3}(s) \big(\boldsymbol I + \boldsymbol{ \mathcal O}(1/n) \big) \boldsymbol D(s)
\] 
for \( s\in\partial U_1 \). Thus, if the matrix
\[
\boldsymbol E_{a_1}(z) := \boldsymbol M(z) K_{a_1}^{n\sigma_3}(z)r_{a_1}^{\sigma_3}(z) \sigma_3\boldsymbol S^{-1}\big(n^2\zeta_{a_1}(z)\big)\sigma_3
\]
is holomorphic in \( U_1 \), \hyperref[rhp]{\rhp}(d) is clearly fulfilled. The fact that it has no jumps on \( \Delta_1, \pi(\ualpha),\pi(\ubeta) \) follows from \hyperref[rhn]{\rhn}(b), \eqref{S}, \eqref{Phi-jump}, and the definition of \( K_{a_1}(z) \). Thus, it is holomorphic in \( U_1\setminus\{a_1\} \). Since $|r_{a_1}(z)|\sim|z-a_1|^{\alpha_1/2}$, \( \boldsymbol S^{-1}\big(n^2\zeta_{a_1}(z)\big)\sim |z-a_1|^{\sigma_3/4} \), and \( \boldsymbol M(z) \) satisfies \hyperref[rhn]{\rhn}(c) around \( a_1 \), the desired claim follows.

\subsection{Approximate Local Parametrix around the Origin}
\label{sec:4.5}

Let \( 0<\delta\leq\delta_0 \), see Section~\ref{ss:ol}. We can assume that the closure of \( U_\delta:=\{|z| < \delta\} \) is disjoint from \( \pi(\ualpha),\pi(\ubeta) \). In this section we construct an approximate solution of \hyperref[rhx]{\rhx} in \( U_\delta \) when \( \ell<\infty \) and an exact solution of \hyperref[rhx]{\rhx} in \( U_\delta \) when \( \ell=\infty \). 

To this end, let functions \( b_i(z) \), \( i\in\{1,2,3,4\} \), be defined in \( \overline U_{\delta_0} \) by
\begin{equation}
\label{bees}
b_1 := \frac{\rho_1+\rho_2}{\rho_2}, \; b_2 := -\frac{\rho_2+\rho_3}{\rho_4}, \; b_3 := -\frac{\rho_3+\rho_4}{\rho_2}, \; \text{and} \; b_4 := \frac{\rho_1+\rho_4}{\rho_4},
\end{equation}
which are holomorphic and non-vanishing on \( \overline U_\delta \). It follows from item (iv) in the definition of class \( \mathcal W_l \) that
\begin{equation}
\label{approximations}
\frac{b_i(0)}{b_i(z)} - 1 =  \mathcal O\big(z^\ell\big) \quad \text{as} \quad z\to0,  \quad i\in\{1,2,3,4\}.
\end{equation}
Notice that \( b_i(z)\equiv b_i(0) \) when \( \ell=\infty \). Observe also that \( b_1(0)=b_3(0) \) and  \( b_2(0)=b_4(0) \) by item (ii) in the definition of class \( \mathcal W_l \). We are seeking a solution of the following local Riemann-Hilbert problem (\rhpo):
\begin{itemize}
	\label{rhpo}
	\item[(a)] $\boldsymbol P_0(z)$ satisfies \hyperref[rhx]{\rhx}(a) within \( U_\delta \);
	\item[(b)] $\boldsymbol P_0(z)$ satisfies \hyperref[rhx]{\rhx}(b)  within \( U_\delta \), where the jump matrix on each \( \tilde\Delta_i^\circ \) needs to be replaced by
\[
\left(\begin{matrix} 1 & 0 \smallskip\\ 	\dfrac{b_i(0)}{b_i(s)}\left(\dfrac{1}{\rho_i(s)} + \dfrac{1}{\rho_{i + 1}(s)}\right) & 1	 \end{matrix}\right);
\]	
\item[(c)] \( \boldsymbol P_0(s) = \boldsymbol M(s)\big(\boldsymbol I+\boldsymbol{\mathcal O}\big((n\delta^2)^{-1/2-|\re(\nu)|} \big)\big)\boldsymbol D(s) \) uniformly for \( s\in\partial U_\delta \) and \( \delta\leq\delta_0 \).
\end{itemize}

\subsubsection{Model Problem}

A construction, similar the one below, has been introduced in \cite{I81}, see also \cite{DZ93} and the book \cite[Chapter~2]{FokasItsKapaevNovokshenov}, in the context of integrable systems. Unfortunately, the local problem is not stated in the form and generality we need in any of these references. Thus, for the convenience of the reader, we provide an explicit expression for the local parametrix.

Let $s_1, s_2 \in \C$ be independent parameters and let $\nu \in \C, \ \re(\nu) \in \left(-\frac{1}{2}, \frac{1}{2} \right]$ be given by 
\begin{equation}
\label{dep-parameters}
e^{-2\pi \ic \nu} := 1 - s_1 s_2
\end{equation}
(we slightly abuse the notation here as the parameter \( \nu \) has already been fixed in \eqref{nu}; however, we shall use the construction below with parameters \( s_1,s_2 \) such that \eqref{dep-parameters} holds with \( \nu \) from \eqref{nu}). Define constants \( d_1,d_2 \) by
\begin{equation}
\label{d1d2}
d_1 := -s_1\dfrac{ \Gamma(1+\nu)}{\sqrt{2\pi}}  \qandq d_2:= -s_2 e^{\pi \ic \nu} \dfrac{\Gamma(1 - \nu)}{\sqrt{2\pi}},
\end{equation}
where \( \Gamma(z) \) is the standard Gamma function. It follows from the well-known Gamma function identities that
\begin{equation}
\label{productd1d2}
d_1d_2 = \ic \nu.
\end{equation}
Denote by $D_\mu(\zeta)$ the parabolic cylinder function in Whittaker's notations, see \cite[Section~12.2]{DLMF}. It is an entire function with the asymptotic expansion
\begin{equation}
\label{pcf-exp}
D_\mu(\zeta) \sim e^{-\zeta^2/4}\zeta^\mu\sum_{k=0}^\infty\frac{(-1)^k}{\Gamma(k+1)}\frac{\Gamma(\mu+1)}{\Gamma(\mu+1-2k)}\frac1{(2\zeta^2)^k}
\end{equation}
valid uniformly in each \( \left|\arg(\zeta) \right| \leq 3\pi/4-\epsilon \), \( \epsilon>0 \), see \cite[Equation~(12.9.1)]{DLMF}. 

Let the matrix function $\boldsymbol \Psi_{s_1, s_2}(\zeta)$ be  given by
\[
\begin{array}{ll}
\left(\begin{matrix}
D_{\nu}(2\zeta) & d_1 D_{-\nu -1} (-2 \ic \zeta) \medskip \\
d_2 D_{\nu - 1}(2\zeta) &  D_{-\nu}(-2\ic \zeta)
\end{matrix}\right)
\left( \begin{matrix} 1 & 0 \medskip \\ 0 & e^{-\pi \ic \nu /2}\end{matrix}\right), & \arg(\zeta) \in \left(0, \frac{\pi}{2} \right), \bigskip \\
\left(\begin{matrix}
D_{\nu}(-2\zeta) & d_1 D_{-\nu -1} (-2 \ic \zeta) \medskip \\
-d_2 D_{\nu - 1}(-2\zeta) & D_{-\nu}(-2\ic \zeta)
\end{matrix}\right)
\left( \begin{matrix} e^{\pi \ic \nu} & 0 \medskip \\ 0 & e^{-\pi \ic \nu/2} \end{matrix}\right), & \arg(\zeta)\in \left(\frac{\pi}{2}, \pi\right), \bigskip \\
\left(\begin{matrix}
D_{\nu}(-2\zeta) & -d_1D_{-\nu -1} (2 \ic \zeta) \medskip \\
-d_2 D_{\nu - 1}(-2\zeta) & D_{-\nu}(2\ic \zeta)
\end{matrix}\right)
\left( \begin{matrix} e^{-\pi \ic \nu} & 0 \medskip \\ 0 & e^{\pi \ic \nu/2} \end{matrix}\right), & \arg(\zeta)\in \left(-\frac{\pi}{2}, -\pi\right), \bigskip \\
\left(\begin{matrix}
D_{\nu} (2\zeta) & - d_1D_{-\nu -1} (2\ic \zeta) \medskip \\
d_2 D_{\nu - 1} (2\zeta) &  D_{- \nu}(2 \ic \zeta)
\end{matrix}\right)
\left( \begin{matrix} 1 & 0 \medskip \\ 0 & e^{\pi \ic \nu /2}\end{matrix}\right),  & \arg(\zeta) \in \left(0, -\frac{\pi}{2} \right).
\end{array}
\]
Then, $\boldsymbol \Psi_{s_1, s_2}(\zeta)$ satisfies the following RH problem (\rhpsiS):
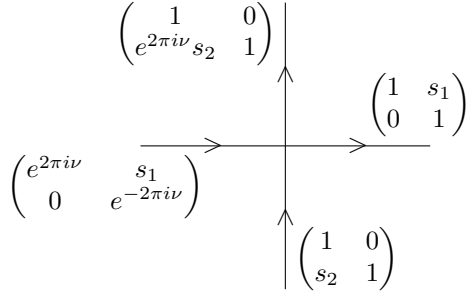
\begin{figure}[h!]
	\begin{tikzpicture}
	[rotate = -135, scale = 0.9]{
		
		\draw (0, 0) -- (1.5, 1.5);
		\node[rotate = -90] at (0.75, 0.75){$<$};
		\node[below right] at (0.75, 0.75){$\left(\begin{matrix}
			1 & 0\\
			s_2 & 1
			\end{matrix}\right)$};
		
		\draw (0, 0) -- (-1.5, -1.5);
		\node[rotate = -90] at (-0.75, -0.75){$<$};
		\node[above left] at (-0.75, -0.75){$\left(\begin{matrix}
			1 & 0\\
			e^{2\pi i \nu}s_2 & 1
			\end{matrix}\right)$};
		
		\draw (0, 0) -- (-1.5, 1.5);
		\node[rotate = 0] at (-0.75, 0.75){$>$};
		\node[above right] at (-0.75, 0.75){$\left(\begin{matrix}
			1 & s_1\\
			0 & 1
			\end{matrix}\right)$};
		
		\draw (0, 0) -- (1.5, -1.5);
		\node[rotate = 0] at (0.75, -0.75){$>$};
		\node[below left] at (0.75, -0.75){$\left(\begin{matrix}
			e^{2\pi i \nu} & s_1\\
			0 & e^{-2\pi i \nu}
			\end{matrix}\right)$};
	}
	\end{tikzpicture}
	\caption{Matrices $\boldsymbol \Psi_{s_1, s_2-}^{-1}\boldsymbol \Psi_{s_1, s_2+}$ on the corresponding rays.}
	\label{f:psiSjumps}
\end{figure}
\begin{itemize}
	\label{rhpsiS}
	\item[(a)] \( \boldsymbol \Psi_{s_1, s_2}(\zeta) \) is analytic in \( \C \setminus \left(\R \cup \ic\R \right) \);
	\item[(b)] \( \boldsymbol \Psi_{s_1, s_2}(\zeta) \) has continuous traces on \( \R \cup \ic\R \) outside of the origin that satisfy the jump relations shown in Figure~\ref{f:psiSjumps};
	\item[(c)] $\boldsymbol \Psi_{s_1, s_2}(\zeta)$ has the following asymptotic expansion as \( \zeta\to\infty \):
	\[
	\left(\boldsymbol I + \frac{1}{2\zeta}\left(\begin{matrix}
	0 & \ic d_1 \\
	d_2 & 0
	\end{matrix}\right) + \frac{ \nu (\nu - 1)}{8\zeta^2}\left(\begin{matrix}
	-1 & 0 \\
	0 & 1
	\end{matrix}\right) + \boldsymbol{\mathcal O}\left(\frac1{\zeta^3}\right)\right) (2\zeta)^{\nu \sigma_3}e^{- \zeta^2 \sigma_3},
	\]
	 which holds uniformly in \( \C \).
\end{itemize}
Indeed, \hyperref[rhpsiS]{\rhpsiS}(a) follows from the fact that $D_{\nu}(\zeta)$ is entire, while \hyperref[rhpsiS]{\rhpsiS}(c) is a consequence of \eqref{pcf-exp}. The jump relations \hyperref[rhpsiS]{\rhpsiS}(b) can be verified using the identities \( \Gamma(-\nu) \Gamma(1 + \nu) = -\pi/\sin (\pi \nu) \), \eqref{dep-parameters}, and
\[
D_{\mu}(2\xi) = e^{-\mu \pi \ic} D_{\mu}(- 2 \xi) + \dfrac{\sqrt{2\pi}}{\Gamma(-\mu)}e^{-(\mu + 1)\pi \ic/2} D_{-\mu -1}(2\ic \xi),
\]
suitably applied with parameter values $\mu = -\nu, \nu - 1 $ and $\xi = \zeta, -\zeta, \ic \zeta$. For later, it will be important for us to make the following observation. Define
\begin{equation}
\label{boldA}
d_\nu := \left\{ \begin{array}{rl} d_2, & \re(\nu)>0, \medskip \\ 0, & \re(\nu)=0, \medskip \\ \ic d_1, & \re(\nu)<0 \end{array}\right. \qandq \boldsymbol A_\nu := \left\{ \begin{array}{rl} \left(\begin{matrix} 0 & 0 \\ 1 & 0 \end{matrix}\right), & \re(\nu)\geq0, \smallskip \\ \left(\begin{matrix} 0 & 1 \\ 0 & 0 \end{matrix}\right), & \re(\nu)<0, \end{array} \right. 
\end{equation}
Recall that we set \( \varsigma_\nu=1,0,-1 \) depending on whether \( \re(\nu)>0 \), \( \re(\nu)=0 \), or \( \re(\nu)<0 \). Observe that 
\begin{multline}
\label{Psic-mod}
\big(\boldsymbol I - (2\zeta)^{-1}d_\nu\boldsymbol A_\nu\big) \boldsymbol\Psi_{s_1, s_2}(\zeta) \\ = (2\zeta)^{\nu\sigma_3} \left(\boldsymbol I + (2\zeta)^{-1-2\varsigma_\nu\nu}d_{-\nu}\boldsymbol A_{-\nu}  + \boldsymbol{\mathcal O}\left(\zeta^{-1-|\varsigma_\nu|}\right)\right) e^{- \zeta^2 \sigma_3}.
\end{multline}

\subsubsection{Conformal Map}

Let, as before, \( \mathcal Q_j \) stand for the \( j \)-th quadrant, $j \in \{1, 2, 3, 4 \}$. Set
\begin{equation}
\label{zeta-0}
\zeta_{0}(z) := \left((-1)^{j-1}\int_{0}^z\frac{s\dd s}{w(s)}\right)^{1/2}, \quad z\in U_\delta\cap \mathcal Q_j.
\end{equation}
Since \( w(z) \) is bounded at \( 0 \) and satisfies $w_+(s)=-w_-(s)$, $s\in\Delta$, the branch of the square root can be chosen so that the function \( \zeta_0(z) \) is in fact holomorphic in $U_\delta$ with a simple zero at the origin. Without loss of generality we can assume that \( \delta \) is small enough for $\zeta_0(z)$ to be conformal on $\overline U_\delta$. 

Since the integrand $(-1)^{j-1}s \dd s/w(s)$ becomes negative purely imaginary on $\Delta_1\cup\Delta_3$, the square root in \eqref{zeta-0} can be chosen so that \( \arg\big(\zeta_0(z)\big) = -\pi/4 \), \( z\in \Delta_3^\circ \). As we have had some freedom in selecting the arcs $\tilde{\Delta}_i$, we shall choose them so that \( \tilde\Delta_3^\circ \) and \( \tilde\Delta_1^\circ \) are mapped by \( \zeta_0(z) \) into positive and negative reals, respectively, while \( \tilde\Delta_4^\circ \) and \( \tilde\Delta_2^\circ \) are mapped into positive and negative purely imaginary numbers.

\subsubsection{Matrix $\boldsymbol P_0$}

Define the function \( r(z):=r_j(z) \), \( z\in\mathcal Q_j \), where we let
\begin{equation}
\label{rjs}
r_1:= \ic e^{\pi\ic\nu}\sqrt{\rho_1}, \ r_2:= \ic e^{-\pi\ic\nu}\frac{\rho_2}{\sqrt{\rho_1}}, \ r_3:= -\ic e^{-\pi\ic\nu}\frac{\rho_4}{\sqrt{\rho_1}}, \   r_4:= -\ic e^{-\pi\ic\nu}\sqrt{\rho_1}
\end{equation}
for a fixed determination of \( \sqrt{\rho_1(z)} \). Furthermore, let 
\begin{equation}
\label{J}
\boldsymbol J(z) := \left\{ \begin{array}{rl}
e^{2\pi \ic \nu \sigma_3}, & \arg z \in \left(-\frac{\pi}{2},0\right), \medskip \\
\left(\begin{matrix}0 & 1 \\ -1 & 0 \end{matrix}\right) e^{2\pi \ic \nu \sigma_3}
, & \arg z \in \left(0, \frac{\pi}{4}\right),\medskip \\ 
\left(\begin{matrix}0 & 1 \\ -1 & 0 \end{matrix}\right)
, & \arg z \in \left(\frac{\pi}{4}, \frac{\pi}{2}\right) \cup \left(- \frac{\pi}{2}, -\pi\right), \medskip \\
\boldsymbol I, & \arg z \in\left(\frac\pi2,\pi\right).
\end{array} \right.
\end{equation}
Finally, recalling \eqref{bees}, put
\begin{equation}
\label{s1s2}
s_1 := b_1(0)=b_3(0) \qandq s_2 := b_2(0)=b_4(0).
\end{equation}
Notice that since \( (\rho_1+\rho_2+\rho_3+\rho_4)(0)=0 \), the parameters \( s_1,s_2 \) satisfy \eqref{dep-parameters} with \( \nu \) given by \eqref{nu}. Then
\begin{equation}
\label{P0}
\boldsymbol P_0(z) := \boldsymbol E_0(z) \boldsymbol \Psi_{s_1, s_2} \big( n^{1/2} \zeta_0(z) \big) \boldsymbol J^{-1}(z) r^{-\sigma_3}(z)
\end{equation}
satisfies \hyperref[rhpo]{\rhpo}(a,b) for any matrix $\boldsymbol E_0(z)$ holomorphic in $U_\delta$. Indeed, \hyperref[rhpo]{\rhpo}(a) is an immediate consequence of \hyperref[rhpsiS]{\rhpsiS}(a). It further follows from \hyperref[rhpsiS]{\rhpsiS}(b) that the jumps of \( \boldsymbol P_0(z) \) are as on Figure~\ref{fig:jP0}.
\begin{figure}[h!]

\begin{tikzpicture}
	
	\draw (0, 0) -- (-2, 0);
	\node at (-1, 0){$>$};
	\node[left] at (-2, 0){$\left(\begin{matrix}
		0 & r_2r_3 \\
		-1/r_2r_3 & 0
		\end{matrix}\right)$};
	
	\draw (0, 0) -- (2, 0);
	\node at (1, 0){$<$};
	\node[right] at (2, 0){$\left(\begin{matrix}
		0 & r_1r_4\\
		-1/r_1r_4 & 0
		\end{matrix}\right)$};
	
	\draw (0, 0) -- (2, 2);
	\node[rotate = 45] at (1, 1){$<$};
	\node[above right] at (2, 2){$\left(\begin{matrix}
		1 & 0\\
		-s_1e^{2\pi\ic\nu}/r_1^2 & 1
		\end{matrix}\right)$};
	
	\draw (0, 0) -- (-2, -2);
	\node[rotate = 45] at (-1, -1){$>$};
	\node[below left] at (-2, -2){$\left(\begin{matrix}
		1 & 0\\
		s_1/r_3^2 & 1
		\end{matrix}\right)$};
	
	\draw (0, 0) -- (-2, 2);
	\node[rotate = -45] at (-1, 1){$>$};
	\node[above left] at (-2, 2){$\left(\begin{matrix}
		1 & 0\\
		s_2/r_2^2 & 1
		\end{matrix}\right)$};
	
	\draw (0, 0) -- (2, -2);
	\node[rotate = -45] at (1, -1){$<$};
	\node[below right] at (2, -2){$\left(\begin{matrix}
		1 & 0\\
		-s_2e^{-2\pi\ic\nu}/r_4^2 & 1
		\end{matrix}\right)$};
	
	\draw (0, 0) -- (0, -2);
	\node[rotate = 90] at (0, -1){$>$};
	\node[below] at (0, -2){$e^{2\pi\ic\nu\sigma_3}\left(\begin{matrix}
		0 & -r_3r_4\\
		1/r_3r_4 & 0
		\end{matrix}\right)$};
	
	\draw (0, 0) -- ( 0, 2);
	\node[rotate = 90] at (0, 1){$<$};
	\node[above] at (0, 2){$\left(\begin{matrix}
		0 & -r_1r_2\\
		1/r_1r_2 & 0
		\end{matrix}\right)$};
	\end{tikzpicture}
	
	\caption{The jump matrices of \( \boldsymbol P_0(z) \).}
	\label{fig:jP0}
\end{figure}
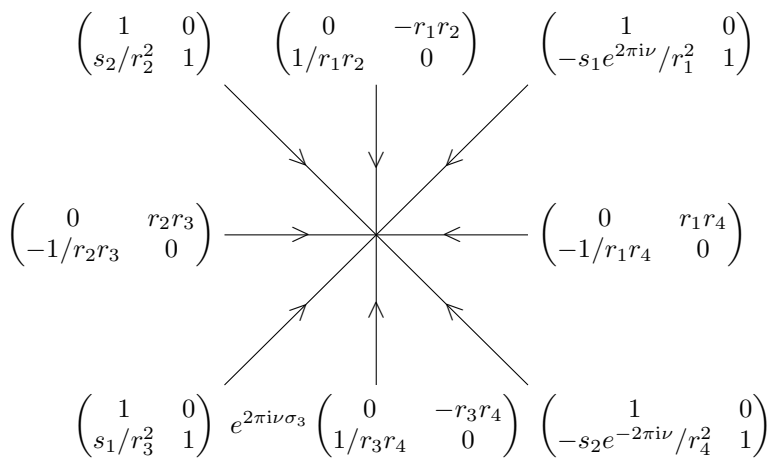
To verify \hyperref[rhpo]{\rhpo}(b), it remains to observe that
\[
r_1r_4 = \rho_1, \; -r_1r_2 = \rho_2, \; r_2r_3 = e^{-2\pi\ic\nu}\rho_2\rho_4/\rho_1 = \rho_3, \; -r_3r_4e^{2\pi\ic\nu} = \rho_4,
\]
since \( e^{-2\pi\ic\nu}=(\rho_1\rho_3)/(\rho_2\rho_4) \), and that
\begin{eqnarray*}
-e^{2\pi\ic\nu}\frac{s_1}{r_1^2} &=&  \frac{b_1(0)}{\rho_1} = \frac{b_1(0)}{b_1}\left( \frac1{\rho_1} + \frac1{\rho_2}\right), \\
\frac{s_2}{r_2^2} &=&  - e^{2\pi\ic\nu}b_2(0)\frac{\rho_1}{\rho_2^2} = -b_2(0)\frac{\rho_4}{\rho_2\rho_3} = \frac{b_2(0)}{b_2}\left( \frac1{\rho_2} + \frac1{\rho_3}\right), \\
\frac{s_1}{r_3^2} &=& - e^{2\pi\ic\nu}b_3(0)\frac{\rho_1}{\rho_4^2} = -b_3(0)\frac{\rho_2}{\rho_3\rho_4} = \frac{b_3(0)}{b_3}\left( \frac1{\rho_3} + \frac1{\rho_4} \right), \\
-e^{-2\pi\ic\nu}\frac{s_2}{r_4^2} &=&  \frac{b_4(0)}{\rho_1} = \frac{b_4(0)}{b_4}\left( \frac1{\rho_1} + \frac1{\rho_4} \right).
\end{eqnarray*}

Thus, it remains to choose $\boldsymbol E_0(z)$ so that \hyperref[rhpo]{\rhpo}(c) is fulfilled.

\subsubsection{Matrix $\boldsymbol E_0$}

Let $\boldsymbol\gamma$ be the part of \( \bd_3 \) that proceeds from \( \boldsymbol a_3 \) to \( {\boldsymbol 0^*} \), see Figures~\ref{f:cross} and \ref{f:basis}. Define
\begin{equation}
\label{K0}
K_0(z) := \left\{ \begin{array}{rl}
\exp \big \{ -\int_{\boldsymbol \gamma} G \big \} = \Phi\left({\boldsymbol 0} \right), & z\in\mathcal Q_1\cup\mathcal Q_3, \medskip \\
\exp \big \{ \int_{\boldsymbol \gamma} G \big \} = \Phi \left( {\boldsymbol 0^*} \right), & z\in\mathcal Q_2\cup\mathcal Q_4.
\end{array} \right.
\end{equation}
 \eqref{K0-formula} immediately yields that $|K_0(z)| \equiv 1$. Define
\begin{equation}
\label{E0star}
\boldsymbol E_0^\star(z) := \boldsymbol M^\star(z) r^{\sigma_3} (z) K_0^{n\sigma_3}(z) \boldsymbol J(z) \zeta_0^{-\nu \sigma_3}(z),
\end{equation}
see \eqref{M}. From \hyperref[rhn]{\rhn}(b), the definition of $\boldsymbol J(z)$, and the fact that $\zeta_0(z)$ maps $\tilde\Delta_1^{\circ}$ into the negative reals, it follows that $\boldsymbol E_0^\star(z)$ is holomorphic in $U_{\delta} \setminus \{ 0 \}$. Furthermore, \hyperref[rhn]{\rhn}(c) combined with the fact that $\zeta_0(z)$ possesses a simple zero at $z=0$ imply that $\boldsymbol E_0^\star(z)$ is holomorphic in $U_{\delta}$. Observe also that the moduli of the entries of \( \boldsymbol E_0^\star(z) \) depend only on the parity of~\( n \).

Put for brevity \( \epsilon_{\nu,n} := (4n)^{\varsigma_\nu\nu-1/2} \), where, as before, \( \varsigma_\nu \) is equal to \( 1,0,-1 \) depending on whether \( \re(\nu) \) is positive, zero, or negative. Set
\begin{equation}
\label{defL}
\boldsymbol L_\nu :=  \frac{d_\nu\epsilon_{\nu,n}}{\zeta^\prime_0(0) D_n}\boldsymbol E_0^{\star}(0) \boldsymbol A_\nu \boldsymbol E_0^{\star-1}(0),
\end{equation}
where \( d_\nu,\bold A_\nu \) were defined in \eqref{boldA} and we assume that
\begin{equation}
\label{Dnnot0}
0 \neq D_n := 1 -d_\nu\epsilon_{\nu,n}\big(\zeta^\prime_0(0)\big)^{-1} E_\nu
\end{equation}
with
\[
E_\nu := \left\{ \begin{array}{rcl} \big[ \boldsymbol E_0^{\star-1}(0)\boldsymbol E_0^{\star\prime}(0)\big]_{12} & \text{if} & \re(\nu)\geq0, \medskip \\ \big[ \boldsymbol E_0^{\star-1}(0)\boldsymbol E_0^{\star\prime}(0)\big]_{21} & \text{if} & \re(\nu)<0. \end{array} \right.
\]
Notice that \( \boldsymbol L_\nu \) is the zero matrix when \( \re(\nu)=0 \) as \( d_\nu=0 \) by \eqref{boldA}. Let
\begin{equation}
\label{E0}
\boldsymbol E_0(z) := (\boldsymbol I + \boldsymbol L_\nu/z)\boldsymbol E_0^\star(z)(4n)^{-\nu\sigma_3/2} \big(\boldsymbol I-d_\nu\big(2n^{1/2}\zeta_0(z)\big)^{-1}\boldsymbol A_\nu\big).
\end{equation}
Let us show that thus defined matrix \( \boldsymbol E_0(z) \) is holomorphic at the origin. Indeed, it has at most double pole there. It is quite simple to see that the coefficient next to \( z^{-2} \) is equal to
\[
-d_\nu\epsilon_{\nu,n}(4n)^{-\varsigma_\nu\nu/2}\big(\zeta^\prime_0(0)\big)^{-1}\boldsymbol L_\nu \boldsymbol E_0^\star(0)\boldsymbol A_\nu,
\]
which is equal to the zero matrix since \( \boldsymbol A_\nu^2 \) is equal to the zero matrix. Using this observation we also get that the coefficient next to \( z^{-1} \) is equal to
\[
\boldsymbol L_\nu \boldsymbol E_0^\star(0)(4n)^{-\nu\sigma_3/2} - d_\nu\epsilon_{\nu,n}(4n)^{-\varsigma_\nu\nu/2}\big(\zeta^\prime_0(0)\big)^{-1}\big(\boldsymbol E_0^\star(0)+\boldsymbol L_\nu \boldsymbol E_0^{\star\prime}(0)\big)\boldsymbol A_\nu,
\]
which simplifies to
\[
\frac{d_\nu\epsilon_{\nu,n}(4n)^{-\varsigma_\nu\nu/2}}{\zeta^\prime_0(0)D_n} \left(1 -\frac{d_\nu\epsilon_{\nu,n}}{\zeta^\prime_0(0)}E_\nu - D_n \right) \boldsymbol E_0^\star(0)\boldsymbol A_\nu,
\]
that is equal to the zero matrix by the very definition of \( D_n \).

Now, recalling the definition of $\Phi(\z)$ in \eqref{Phi} and of \( \zeta_0(z) \) in \eqref{zeta-0}, one can see that
\begin{equation}
\label{zeta0-phi}
\exp \left \{-\zeta_0^2(z) \right \} = e^{-\int_{\boldsymbol \gamma} G} \left\{ \begin{array}{ll}
\Phi\big(z^{(1)}\big), & z\in\mathcal Q_1\cup\mathcal Q_3, \medskip \\
\Phi\big(z^{(0)}\big), &  z\in\mathcal Q_2\cup\mathcal Q_4.
\end{array} \right.
\end{equation}
In particular, since $\boldsymbol D(z)=\Phi^{n\sigma_3}\big(z^{(0)}\big)$ and \( \Phi\big(z^{(0)}\big)\Phi\big(z^{(1)}\big)\equiv 1 \), it follows from \eqref{K0} that
\[
\exp \left \{-n\zeta_0^2(z)\sigma_3 \right \} \boldsymbol J^{-1}(z) = \boldsymbol J^{-1}(z) K_0^{-n\sigma_3}(z) \boldsymbol D(z). 
\]
For brevity, let \( \boldsymbol H(z) := r^{\sigma_3} (z) K_0^{n\sigma_3}(z) \boldsymbol J(z) \). Then we get from \eqref{Psic-mod} and the previous identity that
\begin{multline*}
\boldsymbol E_0(s)\boldsymbol\Psi_{s_1,s_2}\big(n^{1/2}\zeta_0(s) \big)\boldsymbol J^{-1}(s)r^{-\sigma_3}(s) = \\
\boldsymbol M(s) \boldsymbol H(s)\left(\boldsymbol I + \boldsymbol{\mathcal O} \left( \big( n\zeta_0^2(s) \big)^{-1/2-|\re(\nu)|} \right)\right)\boldsymbol H^{-1}(s) \boldsymbol D(s) = \\  \boldsymbol M(s) \left(\boldsymbol I + \boldsymbol{\mathcal O}\left( \big( n\delta^2\big)^{-1/2-|\re(\nu)|} \right)\right)\boldsymbol D(s).
\end{multline*}

It remains to show that \eqref{Dnnot0} holds for all \( n\in\N_{\rho,\varepsilon} \). It follows from \eqref{Arhon} that it is enough to show that
\begin{equation}
\label{target}
A_{\rho,n} = d_\nu\epsilon_{\nu,n}\big(\zeta^\prime_0(0)\big)^{-1} E_\nu.
\end{equation}

\subsection{Existence of \( \boldsymbol L_\nu \)}
\label{sec:5.2}

Assume that \( \re(\nu)> 0 \). It can be readily verified that
\[
E_\nu = \gamma_n\gamma_{n-1}^*  \big(  [\boldsymbol E_0^{\star\prime}(0)]_{12}[\boldsymbol E_0^{\star}(0)]_{22} -[\boldsymbol E_0^{\star\prime}(0)]_{22}[\boldsymbol E_0^{\star}(0)]_{12}\big),
\]
where we used the fact that \( \det(\boldsymbol E_0^\star(z)) = \det(\boldsymbol M^\star(z)) = (\gamma_n\gamma_{n-1}^*)^{-1} \). Notice that \( d_2\neq 0 \) by \eqref{productd1d2}. Using \eqref{E0star}, \eqref{J}, and \eqref{K0} gives us that \( [\boldsymbol E_0^\star(z)]_{i2}  \) is equal to
\[
\zeta_0^\nu(z)\Phi^n\big({\boldsymbol 0}\big) \left\{
\begin{array}{rl}
e^{-2\pi\ic\nu}r_1(z)[\boldsymbol M^\star(z)]_{i1}, & \arg(z)\in(0,\pi/4), \medskip \\
r_1(z)[\boldsymbol M^\star(z)]_{i1}, & \arg(z)\in(\pi/4,\pi/2), \medskip \\
~[\boldsymbol M^\star(z)]_{i2}/r_2(z), & \arg(z)\in(\pi/2,\pi), \medskip \\
r_3(z)[\boldsymbol M^\star(z)]_{i1}, & \arg(z)\in(\pi,3\pi/2), \medskip \\
e^{-2\pi\ic\nu}[\boldsymbol M^\star(z)]_{i2}/r_4(z), & \arg(z)\in(3\pi/2,2\pi).
\end{array}
\right.
\]
Define
\[
S(z) := \zeta_0^\nu(z) \left\{
\begin{array}{rl}
e^{-2\pi\ic\nu}r_1(z)S_\rho\big(z^{(0)}\big), & \arg(z)\in(0,\pi/4), \medskip \\
r_1(z)S_\rho\big(z^{(0)}\big), & \arg(z)\in(\pi/4,\pi/2), \medskip \\
S_\rho\big(z^{(1)}\big)/(r_2w)(z), & \arg(z)\in(\pi/2,\pi), \medskip \\
r_3(z)S_\rho\big(z^{(0)}\big), & \arg(z)\in(\pi,3\pi/2), \medskip \\
e^{-2\pi\ic\nu}S_\rho\big(z^{(1)}\big)/(r_4w)(z), & \arg(z)\in(3\pi/2,2\pi),
\end{array}
\right.
\]
which is a holomorphic and non-vanishing function around the origin. Then we obtain from \eqref{M}, \eqref{product-Psis}, and \eqref{T-ratio} that
\begin{equation}
\label{target1}
E_\nu = S^2(0)\Phi^{2n}\big({\boldsymbol 0}\big)Y_nX_n^{-1}.
\end{equation}

When \( |\pi(\z_k)|=\infty \), the first condition in the definition of \( \N_{\rho,\varepsilon} \) implies that we are looking only at those indices \( n \) for which \( \z_{\imath(n)}=\infty^{(1)} \). In this case \( A_{\rho,n}=0 \) by its very definition in \eqref{Arhon} and it also follows from Lemma~\ref{lem:4.9} that \( Y_n=0 \) in this case. Hence, \eqref{target} does hold in this case.

Let now \( |\pi(\z_k)|<\infty \) and therefore the first condition in the definition of \( \N_{\rho,\varepsilon} \) is void.  It follows from \eqref{zeta-0} and \eqref{w} as well as the fact that \( \zeta_0(z) \) maps \( \{\arg(z)=5\pi/4 \} \) into the positive reals that
\begin{equation}
\label{target2}
1/\zeta_0^\prime(0) = e^{5\pi\ic/4}\sqrt{2ab}.
\end{equation}
Since \( e^{-2\pi\ic\nu}=(\rho_1\rho_3)(0)/(\rho_2\rho_4)(0) \) by \eqref{nu}, we get from \eqref{rjs} that
\begin{equation}
\label{target3}
S^2(0) = -\big(\rho_3\rho_4/\rho_2\big)(0)(2ab)^{-\nu}\lim_{z\to0,~\arg(z)=5\pi/4}|z|^{2\nu}S^2_\rho\big(z^{(0)}\big).
\end{equation}
Observe also that 
\begin{equation}
\label{target4}
d_2 = e^{\pi\ic\nu}\frac{(\rho_2+\rho_3)(0)}{\rho_4(0)}\frac{\Gamma(1-\nu)}{\sqrt{2\pi}}
\end{equation}
by \eqref{d1d2}, \eqref{s1s2}, and \eqref{bees}. Then it follows from \eqref{XnYn} and the very definitions of \( A_{\rho,n} \) in \eqref{Arhon} that \eqref{target1}--\eqref{target4} yield \eqref{target}. The proof of \eqref{target} in the case \( \re(\nu)<0 \) is similar.

Since \( |\pi(\z_k)|<\infty \), the quantities \( Y_n \) and \( Z_n \) in \eqref{T-ratio} and \eqref{T-ratio1} are non-zero and equal to
\[
W_{\imath(n)}^\prime(\boldsymbol o)\frac{T_{\imath(n-1)}^2(\boldsymbol o)}{\Phi(\boldsymbol o)}, \quad W_{\imath(n)}(\z):=\frac{T_{\imath(n)}(\z)}{T_{\imath(n-1)}(\z)},
\]
where \( \boldsymbol o \) was defined in \eqref{sigma-o}. Hence, it follows from \eqref{defL}, \eqref{target}, \eqref{target1}, and a computation similar to the one carried out at the beginning of this subsection  that
\[
\boldsymbol L_\nu = \frac{A_{\rho,n}}{1-A_{\rho,n}}\frac1{W_{\imath(n)}^\prime(\boldsymbol o)}  \left( \begin{matrix} W_{\imath(n)}(\boldsymbol o) & -\Phi(\boldsymbol o)W_{\imath(n)}^2(\boldsymbol o) \medskip \\ 1/\Phi(\boldsymbol o) & -W_{\imath(n)}(\boldsymbol o) \end{matrix}\right).
\]
Moreover, since \( W_1(\z) = 1/W_0(\z) \) we can rewrite the first row of \( \boldsymbol L_\nu \) as
\begin{equation}
\label{L}
\big(\begin{matrix} 1 & 0 \end{matrix}\big)\boldsymbol L_\nu = (-1)^{\imath(n)}\frac{A_{\rho,n}}{1-A_{\rho,n}}\frac{W_0(\boldsymbol o)}{W_0^\prime(\boldsymbol o)} \left( \begin{matrix} 1 & -\Phi(\boldsymbol o)W_{\imath(n)}(\boldsymbol o) \end{matrix} \right).
\end{equation}

\subsection{Final Riemann-Hilbert Problem}

In what follows, we assume that \( \delta=\delta_n\leq\delta_0\) in Section~\ref{sec:4.5} when \( \ell<\infty \) and shall specify the exact dependence on \( n \) later on in this section. When \( \ell=\infty \), we simply take \( \delta=\delta_0 \). Set \( U:=\cup_{i=1}^4 U_{a_i} \) and define
\[ 
\Sigma_n := \big(\partial U\cup \partial U_{\delta_n}\big) \cup \left(\cup_{i=1}^4\big(\Gamma_{i-}\cup\Gamma_{i+}\cup\tilde\Delta_i\big)\setminus \overline U\right),
\] 
see Figure~\ref{fig:Z}.
\begin{figure}[h!]
	\begin{tikzpicture}[rotate = -45, scale =0.8]
	\draw (0, 0) circle (0.5);
	\node[rotate = -90] at (0.3535, 0.3535){$<$};
	\draw[style = dashed] (0, 0) circle (0.85);
	
	
	\node at (-2, 2){\textbullet};
	\node[above] at (-2, 2){$a_2$};
	\draw (-2, 2) circle (0.5);
	\node[rotate = -45] at (-2, 2.5){$<$};
	
	\node at (-2, -2){\textbullet};
	\node[below] at (-2, -2){$a_3$};
	\draw (-2, -2) circle (0.5);
	\node[rotate = -45] at (-2, -2.5){$>$};
	
	\node at (2, 2){\textbullet};
	\node[below] at (2, 2){$a_1$};
	\draw (2, 2) circle (0.5);
	\node[rotate = -45] at (2, 2.5){$<$};
	
	\node at (2, -2){\textbullet};
	\node[below] at (2, -2){$a_4$};
	\draw (2, -2) circle (0.5);
	\node[rotate = -45] at (2, -2.5){$>$};
	
	\draw (0.7, 0) to[out = 0, in = 270](2, 1.5);
	\draw (0.7, 0) to[out = -0, in = -270](2, -1.5);
	
	\draw (-0.7, 0) to[out = 180, in = 270](-2, 1.5);
	\draw (-0.7, 0) to[out = 180, in = -270](-2, -1.5);
	\draw (0.7, 0)--(-0.7, 0);
	
	\draw (0, -0.7) to[out = 270, in = 180](1.5, -2);
	\draw (0, -0.7) to[out = 270, in = 0](-1.5, -2);
	\draw (0, -0.7)--(0, 0.7);
	
	\draw (0, 0.7) to[out = 90, in =180](1.5, 2);
	\draw (0, 0.7) to[out = 90, in = 0](-1.5, 2);
	
	\node at (2.5, -1){$\Gamma^-_{4}$};
	\node at (1, -2.5){$\Gamma^+_{4}$};
	
	\node[rotate = -90] at (0.5, -1.67){$<$};
	\node[rotate = -90] at (1.67, -0.5){$<$};
	
	\node[rotate = -90] at (-0.5, 1.67){$>$};
	\node[rotate = -90] at (-1.67, 0.5){$>$};
	
	\node[rotate = 0] at (0.5, 1.67){$<$};
	\node[rotate = 0] at (1.67, 0.5){$<$};
	
	\node[rotate = 0] at (-0.5, -1.67){$>$};
	\node[rotate = 0] at (-1.67, -0.5){$>$};
	
	\end{tikzpicture}
	\caption{Contour $\Sigma_n$ for \hyperref[rhz]{\rhz} (dashed circle represents \( \{|z|=\delta_0\} \)).}
	\label{fig:Z}
\end{figure}
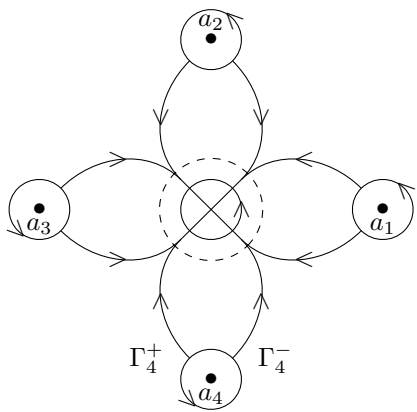
We are looking for a solution of the following Riemann-Hilbert problem (\rhz):
\begin{itemize}
\label{rhz}
\item[(a)] \( \boldsymbol Z(z) \) is analytic in \( \overline\C\setminus \Sigma_n \) and \( \lim_{z\to\infty} \boldsymbol Z(z)= \boldsymbol I\);
\item[(b)] \( \boldsymbol Z(z) \) has continuous traces outside of non-smooth points of \( \Sigma_n \) that satisfy 
\[
\boldsymbol Z_+ = \boldsymbol Z_-
\left\{ \begin{array}{ll}
\boldsymbol P_{a_i}(\boldsymbol{MD})^{-1}, & \text{on } \partial U_{a_i}, \medskip \\
\boldsymbol P_0 (\boldsymbol{MD})^{-1}, & \text{on } \partial U_\delta, \medskip \\
\boldsymbol{MD} \left( \begin{matrix} 1 & 0 \\ 1/\rho_i & 1 \end{matrix} \right)(\boldsymbol{MD})^{-1}, & \text{on } \big(\Gamma_{i + }^\circ \cup \Gamma_{i - }^\circ \big)\setminus \overline U,
\end{array}
\right.
\]
and
\[
\boldsymbol Z_+ = \boldsymbol Z_-
\left\{ \begin{array}{ll}
\boldsymbol{MD} \left( \begin{matrix} 1 & 0 \\ \frac{\rho_i+\rho_{i+1}}{\rho_i\rho_{i+1}} & 1 \end{matrix} \right)(\boldsymbol{MD})^{-1}, & \text{on } \tilde \Delta_i^\circ \setminus \overline U_{\delta_n}, \medskip \\
\boldsymbol P_{0-} \left( \begin{matrix} 1 & 0 \\ \frac{\rho_i+\rho_{i+1}}{\rho_i\rho_{i+1}} & 1 \end{matrix} \right)\boldsymbol P_{0+}^{-1}, & \text{on } \tilde \Delta_i^\circ \cap U_{\delta_n}
\end{array}
\right.
\]
(notice that the second set of jumps is not present when \( \ell=\infty \) as \( \delta_n=\delta_0 \) and \( \boldsymbol P_0(z) \) is the exact parametrix).
\end{itemize}

It follows from \hyperref[rhp]{\rhpi}(d) that the jump of \( \boldsymbol Z \) on \( \partial U_{a_i} \) can be written as
\[
\boldsymbol M(s) \big(\boldsymbol I+\boldsymbol{\mathcal O}(1/n) \big) \boldsymbol M^{-1}(s) = \boldsymbol I+\boldsymbol{\mathcal O}_\varepsilon(1/n)
\]
since the matrix \( \boldsymbol M(z) \) is invertible (its determinant is equal to the reciprocal of \( \gamma_n\gamma_{n-1}^*\)), the matrix \( \boldsymbol M^\star(z) \) depends only on the parity of \( n \), see \eqref{M}, and the matrix \( \boldsymbol L_\nu \) has trace and determinant zero as well as bounded entries for all \( n\in\N_{\rho,\varepsilon} \) and each fixed \( \varepsilon>0 \), see \eqref{defL}. Similarly, we get from \hyperref[rhpo]{\rhpo}(c) that the jump of \( \boldsymbol Z \) on \( \partial U_{\delta_n} \) can be written as
\begin{multline*}
\boldsymbol M(s)\left(\boldsymbol I+\boldsymbol{\mathcal O}\left(\big(n\delta_n^2\big)^{-1/2-|\re(\nu)|}\right)\right)\boldsymbol M^{-1}(s) \\ =\boldsymbol I + \big( \boldsymbol I+\boldsymbol L_\nu/s\big)\boldsymbol{\mathcal O}\left(\big(n\delta_n^2\big)^{-1/2-|\re(\nu)|}\right)\big( \boldsymbol I-\boldsymbol L_\nu/s\big),
\end{multline*}
where \( \boldsymbol{\mathcal O}(\cdot) \) does not depend on \( n \). Since \( \boldsymbol L_\nu = \boldsymbol{\mathcal O}_\varepsilon\big(n^{|\re(\nu)|-1/2}\big) \) by its very definition in  \eqref{defL}, we get that the jump of \( \boldsymbol Z \) on \( \partial U_{\delta_n} \) can further be written as
\[
\boldsymbol I + \boldsymbol{\mathcal O}_\varepsilon\left(\big(n\delta_n^2\big)^{-1/2-|\re \nu|}\max\big\{1,n^{2|\re(\nu)|}/(n\delta_n^2)\big\}\right).
\]

One can easily check with the help of \eqref{N} and \eqref{M} that the jump of \( \boldsymbol Z \) on $\big(\Gamma_{i +}^\circ \cup \Gamma_{i -}^\circ\big)\setminus\overline U$ is equal to
\begin{multline*}
\boldsymbol I + \frac{\gamma_n\gamma_{n-1}^*}{(w^2\rho_i)(s)}\big( \boldsymbol I+\boldsymbol L_\nu/s\big)\left( \begin{matrix} (\Psi_n\Psi_{n-1})\big(s^{(1)}\big) & -\Psi_n^2\big(s^{(1)}\big) \medskip \\ \Psi_{n-1}^2\big(s^{(1)}\big) & -(\Psi_n\Psi_{n-1})\big(s^{(1)}\big) \end{matrix} \right) \big( \boldsymbol I-\boldsymbol L_\nu/s\big) \\ = \boldsymbol I + \boldsymbol{\mathcal O}_\varepsilon(e^{-cn})
\end{multline*}
 for some constant \( c>0 \) by \eqref{Psin} and since the maximum of \( |\Phi(s^{(1)})| \) on \( \Gamma_{i\pm}\setminus U \) is less than \(1\). The estimate of the jump of \( \boldsymbol Z \) on \( \tilde \Delta_i^\circ \setminus \overline U_{\delta_n} \) is analogous and yields 
 \[
\boldsymbol I + \boldsymbol{\mathcal O}_\varepsilon\left(e^{-cn\delta_n^2}\max\big\{1,n^{2|\re(\nu)|}/(n\delta_n^2)\big\}\right)
\]
for an adjusted constant \( c> 0 \), where the rate estimate follows from \eqref{zeta0-phi} as
\[
\big|\Phi\big(s^{(1)}\big)\big| = \exp\left\{(-1)^i\re\big(\zeta_0^2(s)\big)\right\} = \mathcal O\big(e^{-c\delta_n^2}\big), \quad s\in\tilde \Delta_i \setminus U_{\delta_n},
\]
since \( \zeta_0(z) \) is real on \( \tilde\Delta_1\cup\tilde\Delta_3 \) and is purely imaginary on  \( \tilde\Delta_2\cup\tilde\Delta_4 \).

Finally, it holds on \( \tilde \Delta_i^\circ \cap U_{\delta_n} \) that the jump of \( \boldsymbol Z \) is equal to
\begin{multline*}
\boldsymbol I + \left(1 - \frac{b_i(0)}{b_i(z)} \right) \frac{(\rho_i+\rho_{i+1})(s)}{(\rho_i\rho_{i+1})(s)} \boldsymbol P_{0+}(s)\left( \begin{matrix} 0 & 0 \\ 1 & 0 \end{matrix} \right)\boldsymbol P_{0+}^{-1}(s) = \\ \boldsymbol I + \mathcal O(\delta_n^\ell)\boldsymbol E_0(s) \left( \begin{matrix} [\boldsymbol\Psi_+(s)]_{1j}[\boldsymbol\Psi_+(s)]_{2j} & -[\boldsymbol\Psi_+(s)]_{1j}^2 \medskip \\ [\boldsymbol\Psi_+(s)]_{2j}^2 & -[\boldsymbol\Psi_+(s)]_{1j}[\boldsymbol\Psi_+(s)]_{2j} \end{matrix} \right) \boldsymbol E_0^{-1}(s)
\end{multline*}
by \eqref{approximations} and \eqref{P0}, where \( j=1 \) for \( s\in\tilde\Delta_1\cup\tilde\Delta_3 \) and \( j=2 \) for \( s\in\tilde\Delta_2\cup\tilde\Delta_4 \), and we set for brevity \( \boldsymbol\Psi(z):=\boldsymbol\Psi_{s_1,s_2}\big(n^{1/2}\zeta_0(z)\big) \) (observe also that \( \det(\boldsymbol\Psi(z))\equiv1 \)). It follows from the asymptotic expansion \eqref{pcf-exp} that \( D_\mu(x) \) is bounded for \( x\geq0 \). Thus, we deduce from the definition of \( \boldsymbol \Psi(z) \) that the above jump matrix can be estimated as
\[
\boldsymbol I + \mathcal O(\delta_n^\ell)\boldsymbol E_0(s)\boldsymbol{\mathcal O}(1)\boldsymbol E_0^{-1}(s) = \boldsymbol I + \boldsymbol{\mathcal O}_\varepsilon\left( n^{|\re(\nu)|} \delta_n^\ell \right),
\]
where the last equality follows from \eqref{E0star} and \eqref{E0} as \( \boldsymbol E_0(z) \) is equal to a bounded matrix that depends only on \( \epsilon_{\nu,n} \) multiplied by \( (4n)^{\nu\sigma_3/2} \) on the right.

When \( \ell \geq 4|\re(\nu)|(1+|\re(\nu)|)/(1-2|\re(\nu)|) \), choose
\begin{equation}
\label{deltan}
\delta_n = \delta_0\exp\left\{-\frac12\frac{1+4|\re(\nu)|}{\ell+1+2|\re(\nu)|}\ln n\right\}.
\end{equation}
Then it holds that \( n^{2|\re(\nu)|}/(n\delta_n^2) = \mathcal O(1) \) and
\[
n^{|\re(\nu)|}(\delta_n/\delta_0)^\ell = \big(n(\delta_n/\delta_0)^2\big)^{-|\re(\nu)|-1/2} = n^{-d_{\nu,\ell}}
\]
with \( d_{\nu,\ell} \) defined in \eqref{dnuell}. Otherwise, take
\[
\delta_n = \delta_0\exp\left\{-\frac12\frac{3}{\ell+3+2|\re(\nu)|}\ln n\right\}.
\]
In this case \( n^{2|\re(\nu)|}/(n\delta_n^2)\to\infty \) as \( n\to\infty \) and
\[
n^{|\re(\nu)|}(\delta_n/\delta_0)^\ell = n^{2|\re(\nu)|}\big(n(\delta_n/\delta_0)^2\big)^{-|\re(\nu)|-3/2} = n^{-d_{\nu,\ell}}.
\]

Since \( d_{\nu,\ell}<1 \), it holds that the jumps of \( \boldsymbol Z \) on \( \Sigma_n \) are of order \( \boldsymbol I + \boldsymbol{\mathcal O}_\varepsilon( n^{-d_{\nu,\ell}}) \), where \( \boldsymbol{\mathcal O}_\varepsilon(\cdot) \) does not depend on \( n \). Then, by arguing as in \cite[Theorem~7.103 and Corollary~7.108]{Deift} we obtain that the matrix \( \boldsymbol Z \) exists for all \( n\in\N_{\rho,\varepsilon} \) large enough and that
\[
\|\boldsymbol Z_\pm-\boldsymbol I\|_{2,\Sigma_n} = \mathcal O_\varepsilon\big( n^{-d_{\nu,\ell}}\big).
\]
Since the jumps of \( \boldsymbol Z \) on \( \Sigma_n \) are restrictions of holomorphic matrix functions, the standard deformation of the contour technique and the above estimate yield that
\begin{equation}
\label{z-rate}
\boldsymbol Z = \boldsymbol I + \boldsymbol{\mathcal{O}}_\varepsilon\big( n^{-d_{\nu,\ell}}\big)
\end{equation}
locally uniformly in \( \overline\C\setminus\{0\} \).

\subsection{Proofs of Theorems~\ref{thm:asymptotics} and \ref{thm:pade}}

Given \( \boldsymbol Z(z) \),  a solution of \hyperref[rhz]{\rhz}, \( \boldsymbol P_{a_i}(z) \) and \( \boldsymbol P_0(z) \), defined in \eqref{Pa} and \eqref{P0}, respectively, and \( \boldsymbol C(\boldsymbol{MD})(z) \)  from \eqref{M} and \eqref{CD}, it can be readily verified that 
\begin{equation}
\label{X}
\boldsymbol X(z) :=  \boldsymbol C\boldsymbol Z(z) \left\{  \begin{array} {ll}
\boldsymbol P_{a_i}(z), & z \in U_i, \ i \in\{1,2 ,3 ,4\}, \smallskip \\
\boldsymbol P_0(z), & z \in U_{\delta}, \smallskip \\
(\boldsymbol {MD})(z), & \text{otherwise},
\end{array} \right.
\end{equation}
solves  \hyperref[rhx]{\rhx}. Given a closed set \( K\subset\overline\C\setminus\Delta  \), the contour \( \Sigma_n \) can always be adjusted so that \( K \) lies in the exterior domain of \( \Sigma_n \). Then it follows from \eqref{eq:x} that \( \boldsymbol Y(z)=\boldsymbol X(z) \) on \( K \). Formulae \eqref{Qnasymp} and \eqref{Rnasymp} now follow immediately from \eqref{Y-matrix}, \eqref{eq:x}, \eqref{M}, \eqref{CD}, and \eqref{Psin} since
\[
w^{i-1}(z)[(\boldsymbol{ZMD})(z)]_{1i} = (1 + \upsilon_{n1}(z))\Psi_n\big(z^{(i-1)}\big) + \upsilon_{n2}(z)\Psi_{n-1}\big(z^{(i-1)}\big),
\]
where \( 1+\upsilon_{n1}(z),  \upsilon_{n2}(z) \) are the first row entries of \( \boldsymbol Z(z)(\boldsymbol I+\boldsymbol L_\nu/z) \). Estimates \eqref{upsilons} are direct consequence of \eqref{defL} and \eqref{z-rate}. Relations \eqref{Lni} follow from \eqref{L}. Similarly, if \( K \) is a compact subset of \( \Delta^\circ \), the lens \( \Sigma_n \) can be arranged so that \( K \) does not intersect \( \overline U\cup \overline U_{\delta_n} \). As before, we get that
\begin{multline*}
[(\boldsymbol{ZMD})(z)]_{11} =  \left((1 + \upsilon_{n1}(z))\Psi_n\big(z^{(0)}\big) + \upsilon_{n2}(z)\Psi_{n-1}\big(z^{(0)}\big)\right) \pm \\ (\rho_i w)^{-1}(z) \left((1 + \upsilon_{n1}(z))\Psi_n\big(z^{(1)}\big) + \upsilon_{n2}(z)\Psi_{n-1}\big(z^{(1)}\big)\right)
\end{multline*}
for \( z\in \Omega_{i\pm} \setminus \big( \overline U\cup \overline U_{\delta_n} \big) \). Formula \eqref{Qnasymp1} now follows by taking the trace of \( [(\boldsymbol{ZMD})(z)]_{11} \) on \( \Delta_{i\pm}\setminus \big( \overline U\cup \overline U_{\delta_n} \big) \) and using \eqref{Psin-jump}.

\subsection{Behavior of \( Q_n(z) \) around the Origin when \( \ell=\infty \) and \( |\re(\nu)|<1/2 \)}
\label{ssec:at0}

Assume that \( \ell=\infty \). In this case \( \delta=\delta_n=\delta_0 \) in \eqref{deltan} is independent of \( n \) and \( \boldsymbol P_0(z) \) is the exact  parametrix (that is, the second group of jumps in \hyperref[rhz]{\rhz}(b) is not present). Assume further that \( |\re(\nu)|<1/2 \).  The definition of the matrix function \( \boldsymbol M(z) \) as \( \left( \boldsymbol I + \boldsymbol L_\nu/z\right)\boldsymbol M^\star(z) \) is absolutely necessary when \( |\re(\nu)|=1/2 \), see \eqref{M}, but can be simplified to \( \boldsymbol M(z)=\boldsymbol M^\star(z) \) when \( |\re(\nu)|<1/2 \). That is, we can take \( \boldsymbol L_\nu \) to be the zero matrix. In this case the error rate in \hyperref[rhpo]{\rhpo}(c) will become \( \boldsymbol{\mathcal O}(n^{|\re(\nu)|-1/2}) \) and the matrix \( \boldsymbol E_0(z) \) will simplify to
\[
\boldsymbol E_0(z) = \boldsymbol M(z)K_0^{n\sigma_3}(z)r^{\sigma_3} (z) \boldsymbol J(z) (2\xi_n)^{-\nu \sigma_3}, \quad \xi_n:=\sqrt n\zeta_0(z),
\]
see \eqref{E0star} and \eqref{E0}. Assume now that \( z \) is in the second quadrant, in which case \( \boldsymbol J = \boldsymbol I \). It then follows from \eqref{K0} and \eqref{zeta0-phi} that \( K_0^n(z) = \Phi^n(z^{(0)})e^{\xi_n^2} \). Thus, we get from \eqref{P0} as well as \eqref{N} and \eqref{CD} that
\[
\boldsymbol P_0(z) =  \boldsymbol E_0(z)\boldsymbol\Psi(\xi_n)r_2^{-\sigma_3}(z), \quad  \boldsymbol E_0(z)=\boldsymbol C^{-1}\boldsymbol N(z) \left( r_2(z)e^{\xi_n^2}/(2\xi_n)^\nu \right)^{\sigma_3},
\]
where we write \( \boldsymbol \Psi(\zeta) \) for \( \boldsymbol \Psi_{s_1,s_2}(\zeta) \). Now, \eqref{Y-matrix} and \eqref{eq:x} yield that
\(
Q_n(z) = [\boldsymbol X(z)]_{11} + \rho_3^{-1}(z)[\boldsymbol X(z)]_{12}
\)
for \( z\in \Omega_{3+} \). Therefore, we get from \eqref{X} that
\[
\gamma_n^{-1}Q_n(s) = \left(\begin{matrix} 1 & 0 \end{matrix}\right)\boldsymbol Z(s)\left( [\boldsymbol P_0(s)]_{1+}+\rho_3^{-1}(s)[\boldsymbol P_0(s)]_{2+}\right)
\]
for \( s\in\Delta_3\cap U_\delta \), where \( [\boldsymbol P_0(z)]_i \) stands for the \( i \)-th column of \( \boldsymbol P_0(z) \). It follows from the analyticity of \( \boldsymbol E_0(z) \) in \( U_\delta \) that
\[
\gamma_n^{-1}Q_n(s) = \left(\begin{matrix} 1 & 0 \end{matrix}\right)\boldsymbol Z(s)\boldsymbol E_0(s)\left(r_2^{-1}(s)[\boldsymbol \Psi(\xi_n)]_1 + r_3^{-1}(s)[\boldsymbol \Psi(\xi_n)]_2\right)
\]
since \( r_2(s)r_3(s)=\rho_3(s) \). Using the expression for \( \boldsymbol E_0(z) \) from above as well as \eqref{N} and \eqref{Psin-jump} we get that
\begin{equation}
\label{QnatZero}
\gamma_n^{-1}Q_n(s) = \left(\begin{matrix} 1 & 0 \end{matrix}\right)\boldsymbol Z(s) \left(\begin{matrix} \Psi_{n+}^{(0)}(s) & \Psi_{n-}^{(0)}(s) \\ \Psi_{n-1+}^{(0)}(s)  & \Psi_{n-1-}^{(0)}(s)  \end{matrix}\right) \left( \begin{matrix} (2\xi_n)^{-\nu}A_\rho(\xi_n) \\ (2\ic\xi_n)^\nu B_\rho(\xi_n) \end{matrix} \right)
\end{equation}
for \( s\in\Delta_3\cap U_\delta \), where, since \(\zeta_0(s)\) has argument \( -\pi/4 \) for \( s\in\Delta_3 \), we set
\[
\left\{
\begin{array}{l}
A_\rho(\zeta) := e^{\zeta^2}\big(D_\nu(2\zeta) + \alpha_\rho D_{-\nu-1}(2\ic\zeta)\big) \medskip \\
B_\rho(\zeta) := e^{-\zeta^2} \big(D_{-\nu}(2\ic\zeta) + \beta_\rho D_{\nu-1}(2\zeta)\big)
\end{array}
\right.
\] 
with \( \alpha_\rho := -e^{\pi\ic\nu/2}d_1(r_2/r_3)(s)=d_1(\rho_4/\rho_2)(s) \), \( \beta_\rho := -e^{-\pi\ic\nu/2}d_2(\rho_2/\rho_4)(s) \) and \( d_1,d_2 \) given by \eqref{d1d2}, which are constants by the definition of \( \mathcal W_\infty \). Recall that \( \left(\begin{matrix} 1 & 0 \end{matrix}\right)\boldsymbol Z(s) \), the first row of \( \boldsymbol Z(s) \), behaves like \( \left(\begin{matrix} 1 +o(1) & o(1) \end{matrix}\right) \), where \( o(1)=\mathcal O(n^{|\re(\nu)|-1/2}) \), in the considered case. Therefore, by multiplying \eqref{QnatZero} out, we can get an asymptotic expression for \( Q_n(s) \) around the origin on \( \Delta_3 \). Clearly, we can get similar expressions on the remaining arcs \( \Delta_1, \Delta_2 \) and \( \Delta_4 \). 

A computation along these lines can be performed in the case \( \re(\nu)=1/2 \), but the resulting formula is even more involved than \eqref{QnatZero}.

\end{document}